\documentclass[a4paper,11pt]{amsart}
\usepackage{amssymb, amsmath, mathrsfs}
\usepackage{mathtools}
\mathtoolsset{showonlyrefs}
\usepackage[inline, final]{showlabels}
\usepackage{tikz-cd}
\usepackage{enumitem}
\usepackage{wrapfig}
\usepackage{tikz}
\usetikzlibrary{arrows}

\pgfdeclarelayer{background}
\pgfsetlayers{background,main}
\definecolor{col1}{RGB}{236,126,112}
\definecolor{col2}{RGB}{177,184,180}
\definecolor{col3}{RGB}{38,73,68}

\usepackage[disable]{todonotes}
\usepackage[
 pdfauthor={Asbjorn C. Nordentoft, Yiannis N. Petridis and Morten S. Risager},
 pdftitle={Small scale equidistribution of Hecke eigenforms at infinity },
 pdfkeywords={},
 pdfproducer={LaTeX2e with hyperref},
 pdfcreator={Pdflatex},
 pdfdisplaydoctitle =true]
 {hyperref}
\usepackage{pdfsync}

\newtheorem{thm}{Theorem}[section]
\newtheorem{prop}[thm]{Proposition}
\newtheorem{lem}[thm]{Lemma}
\newtheorem{cor}[thm]{Corollary}
\newtheorem*{rem*}{Remark}

\theoremstyle{definition}

\newtheorem{rem}{Remark}

\def \e {{\varepsilon}}

\def \R {{\mathbb R}}
\def \H {{\mathbb H}}
\def \N {{\mathbb N}}
\def \W {{V}}
\def \C {{\mathbb C}}

\def \G {\Gamma}
\def \Z {\mathbb{Z}}

\def \pslz  {{\hbox{PSL}_2( {\mathbb Z})} }
\def \slz  {{\hbox{SL}_2( {\mathbb Z})} }
\def \GmodH {{\Gamma\backslash\mathbb H}}
\def \eps{\ensuremath\varepsilon}

\def \supp {{\rm supp\,} }
\DeclareMathOperator{\sgn}{sgn}

\newcommand{\vol}[1]{\hbox{vol}\left( #1 \right)}
\newcommand{\abs}[1]{\left\lvert #1 \right\rvert}
\newcommand{\norm}[1]{\left\lVert #1 \right\rVert}

\newcommand{\inprod}[2]{\left \langle #1,#2 \right\rangle}

\providecommand{\sym}{\operatorname{sym}}

\title[Small scale equidistribution]{Small scale equidistribution of Hecke eigenforms at infinity }
\author[A. Nordentoft]{Asbj\o rn C. Nordentoft}
\address{LAGA, Institut Galil\'{e}e, 99 avenue Jean Baptiste Cl\'{e}ment, 93430 Villetaneuse, France}
\email{acnordentoft@outlook.com}

\author[Y. Petridis]{Yiannis N. Petridis}
\address{Department of Mathematics, University College London, Gower Street, London WC1E 6BT, United Kingdom}
\email{i.petridis@ucl.ac.uk}

\author[M. Risager]{Morten S. Risager}
\address{Department of Mathematical
  Sciences, University of Copenhagen, Universitets\-parken 5, 2100
  Copenhagen \O , Denmark}
\email{risager@math.ku.dk}

\date{\today}

\begin{document}
\maketitle
\begin{abstract}We investigate the equidistribution of Hecke eigenforms for $\pslz$ on sets that are shrink\-ing towards the cusp. We show that at scales finer than  the Planck scale they do \emph{not} equidistribute while at scales more coarse than the Planck scale they equidistribute on a full density subsequence of eigenforms. On a suitable set of test functions we compute the variance showing an interesting transition behavior at half the Planck scale.
\end{abstract}

\section{Introduction}
It is a fundamental consequence of Berry's random wave conjecture \cite{Berry:1977a} that the eigenfunctions of the Laplace operator on a hyperbolic manifold $M=\GmodH$   \lq spread out\rq{} in the large eigenvalue limit.
 For a measure $d\nu'$ on $\GmodH$ and a sufficiently nice function $\psi$ on $\GmodH$ we write \begin{equation}
  \inprod{\psi}{d\nu'}=\int_{\GmodH}\psi(z) d\nu'(z).
\end{equation}
Let $\varphi_\lambda$ be $L^2$-normalized eigenfunctions of the Laplacian with eigenvalue $\lambda$, and
consider the measures \begin{equation}d\mu_\lambda=\abs{\varphi_\lambda}^2d\mu, \quad d \nu=\frac{d \mu}{\vol{\GmodH}}  ,\end{equation} where $d\mu(z)=y^{-2}dxdy$ is the uniform measure on the surface.

The  question about whether  the eigenfunctions  indeed spread out is quantified by the question of whether
\begin{equation}\label{equidistribution}
  \inprod{\psi}{d\mu_\lambda}\to\inprod{\psi}{d\nu},\textrm{ as }\lambda \to \infty
\end{equation}
for a suitable set of test functions $\psi$.

 For the full modular group $\G=\pslz$  with $\varphi_\lambda$ being  Hecke--Maass forms this (and much more) was famously proved by Lindenstrauss \cite{Lindenstrauss:2006a} and Sound\-ara\-ra\-jan \cite{Soundararajan:2010a}. Zelditch \cite{ Zelditch:1994} had previously studied the variance sum
\begin{equation}
  \sum_{\lambda\leq \Lambda} \abs{\inprod{\psi}{d\mu_\lambda}-{\inprod{\psi}{d\mu}}}^2
\end{equation}
providing weak but non-trivial upper bounds on this  to conclude \eqref{equidistribution} for a full density subsequence of $\lambda$, see also \cite{Zelditch:1987, Schubert:2006}.
For the full modular group  Sarnak and Zhao \cite{SarnakZhao:2019} were able to prove asymptotics for the variance sum on a suitable set of test functions, and  Nelson \cite{Nelson:2016d, Nelson:2017a, Nelson:2019} has recently found a way to determine the asymptotics also for arithmetic compact hyperbolic surfaces arising from  maximal orders in quaternion algebras.

It is natural to ask if the equidistribution \eqref{equidistribution} still holds if we allow the support of the test function $\psi$ to shrink as a function of $\lambda$. An interesting special case is when $\psi$ is the indicator function of a hyperbolic ball of radius $R$ with $R$ going to zero as a function of $\lambda$. This is the question of equidistribution in {\lq shrinking sets\rq}, which has been analyzed e.g. by Young \cite[Prop 1.5]{Young:2016a}. The physics literature  seems to suggest  that  equidistribution  holds all the way down to the scale of the de Broglie wavelength, which is of the order of $1/\sqrt{\lambda}$, see also \cite{HejhalRackner:1992a}.  Humphries \cite{Humphries:2018} has shown that  below this threshold, also called the Planck scale,  there are cases where equidistribution does \emph{not} hold (he even shows that equidistribution fails slightly above the Planck scale).

 Humphries and Khan \cite{HumphriesKhan:2020} proved that individual equidistribution holds all the way to the Planck scale, if we  restrict to dihedral forms, which form a very thin set of Maass forms.

It should be noted that ergodic theory methods provide equidistribution in shrinking balls for general negatively curved manifolds but typically only for a slow logarithmic rate, see e.g. \cite{Han:2015,HezariRiviere:2016a}.

On the other hand for the eigenfunctions on the Euclidean torus Granville and Wigman \cite{GranvilleWigman:2017} showed individual equidistribution close to the Planck scale and failure of equidistribution at scales at a small power of $\log$ above  the Planck scale. The equidistribution was previously proved by Lester and Rudnick \cite{LesterRudnick:2017} along a full density subsequence.

\subsection{Mass equidistribution for holomorphic Hecke cusp forms}
We may ask questions analogous to the above if we replace the eigenfunction $\varphi_\lambda $ by $y^{k/2}f(z)$, where $f(z)$ is an $L^2$-normalized  holomorphic cusp form of weight $k$ for $\G=\pslz $.  In fact $y^{k/2}f(z)$ \emph{is} an eigenfunction of the weight $k$ Laplacian $\Delta_k$ for the full modular group with eigenvalue  $-k/2(1-k/2)$, which is the bottom of the spectrum for $\Delta_k$. In analogy with \eqref{equidistribution} Holowinsky and Soundararajan \cite{HolowinskySoundararajan:2010} proved that
\begin{equation}
\mu_f(\psi):=\inprod{\psi}{d \mu_f}\to \inprod{\psi}{d \nu}, \textrm{ as }k\to \infty,
\end{equation}  where \begin{equation}d\mu_f=y^k\abs{f(z)}^2d\mu.\end{equation}

Luo and Sarnak \cite{LuoSarnak:2004a} computed the quantum variance of these measures on the modular surface. More precisely they proved that for a fixed compactly supported function $u$ on $\R_+$ we have
\begin{equation}
  \sum_{2\vert k}u\left(\frac{k-1}{K}\right)\sum_{f\in H_k}L(1, \sym^2{f})\abs{\mu_f(\psi)}^2=B_\omega(\psi, \psi)K+O_{\e,\psi}(K^{1/2+\e}).
\end{equation}

Here $H_k$ is an orthonormal basis of Hecke eigenforms, $L(s, \sym^2 f)$ is the symmetric square $L$-function of $f$, and  $\psi$ is a rapidly decaying smooth function of mean zero whose zero-th Fourier coefficient vanishes  sufficiently high in the cusp, and
 $B_\omega(\psi_1, \psi_2)$ is a Hermitian form diagonalized by Hecke--Maass cusp forms. The eigenvalues of $B_\omega$ are  arithmetically significant: they are $\pi/2$ times the central value of the corresponding $L$-function.

\subsection{Equidistribution on shrinking sets}   The question of equidistribution on shrinking sets in the holomorphic setting was considered by Lester, Matom\"aki, and Radziwi\l \l{} \cite{LesterMatomakiRadziwil-l:2018a}. They proved an effective version in terms of the test function of the result of Holowinsky and Soundararajan, allowing to shrink  the test function  at the rate of a small negative power of $\log k$.

We consider the following variant of the  problem about \lq shrinking sets\rq{}: Let $H$ be positive with $H\ge 1$ and define the set
\begin{equation}
  B_H=\{z\in \GmodH : \Im(z)> H\},
\end{equation} considered to be a shrinking ball around the cusp.
 We study the distribution of compactly supported functions on $B_1$ squeezed into $B_H$ using the operator $M_H$ defined by
\begin{equation}
    M_H\psi(z)=\psi(x+i{y}/{H}).
\end{equation}
This may be formulated in a coordinate-independent way, see Section \ref{squeezing-operator}. Similar shrinking has been considered previously by Ghosh and Sarnak \cite{GhoshSarnak:2012} as well as by  Lester, Matom\"aki, and Radziwi\l \l{} \cite{LesterMatomakiRadziwil-l:2018a}.

We will consider mass equidistribition {\lq}high in the cusp{\rq}, by which  we mean that
\begin{equation}\mu_f(M_{H(k)}\psi)\rightarrow \nu(M_{H(k)}\psi), \end{equation}
as $H(k)$ tends to infinity with $k$. The length scale of $B_H$ is of the order  $H^{-1}$, so we might expect equidistribution to hold all the way down to $H^{-1}\gg k^{-1}$, as this is the order of the de Broglie wavelength of $y^{k/2} f(z)$.

Let $B:=B_1$. We will consider the following class of functions:
\begin{equation}C_{0}^\infty(M, B)=\{ \psi\in C_{0}^\infty(M)\mid \supp \psi \subset B\},\end{equation}
where $C_{0}^\infty(M)$ consists of all smooth functions on $M=\GmodH$ decaying rapidly at the cusp, and such that the zero-th Fourier coefficient vanishes sufficiently high in the cusp.
Given $\psi\in C_{0}^\infty(M, B)$, we investigate upper bounds and asymptotics for
\begin{equation}
  \sum_{2\vert k }u\left(\frac{k-1}{K}\right)\sum_{f\in H_k}L(1, \sym^2{f})\abs{\mu_f(M_{H(k)}\psi)-\nu(M_{H(k)}\psi)}^2,
\end{equation}
where $H(k)=(k-1)^\theta$ for some $0\leq \theta <1$, and $u:\R_+\to \R_{\geq 0}$ is smooth with compact support. It turns out that the asymptotics depends crucially on $\theta$.
\subsection{Mass equidistribution below and above the Planck scale}
We first prove that mass equidistribution fails on shrinking sets around the cusp as above for scales finer than the Planck scale. This is consistent with the above prediction and just reflects the fact that $f$ decays rapidly for $y\gg k$, which comes simply from the Fourier expansion.
\begin{prop}
  Let $\theta \geq 1$, i.e. shrinking  below the Planck scale. Then there exists $\psi\in C_0^\infty(M,B)$ such that $\mu_f(M_{(k-1)^\theta}\psi)=o(\nu(M_{(k-1)^\theta}\psi ))$ and $\nu(M_{(k-1)^\theta}\psi )\neq 0$ as $k\rightarrow \infty$.

  \end{prop}

Secondly we  obtain a power-saving bound for the quantum variance sum for general observables all the way down to the Planck scale. This implies that mass equidistribution holds for a density one subsequence of holomorphic cusp forms.
\begin{thm}\label{ergodicity1}Let $0<\theta<1$ and $\psi\in C_{0}^\infty(M,B)$. Then
 \begin{align}
    \sum_{2\mid k} u\left(\frac{k-1}{K}\right) \sum_{f\in H_k}L(1,\sym^2 f) &\abs{\mu_f(M_{(k-1)^\theta}\psi)-\nu(M_{(k-1)^\theta}\psi)}^2\\&= O_{\psi, u} (K^{2-2\theta-\min{(1/5,1-\theta)+\e}}).
\end{align}

  \end{thm}
Since $\nu(M_{(k-1)^\theta}\psi)$ is of size about $k^{-\theta}$ this supplements the results in \cite[Theorem 1.3]{LesterMatomakiRadziwil-l:2018a} as it shows that equidistribution holds on average at a much finer scale than  individually, as proved in \cite{LesterMatomakiRadziwil-l:2018a}. The precise polynomial saving of 1/5 when $\theta$ is sufficiently small can probably be improved; its proof has as its input the convexity bound in the $k$-aspect of $L(s, \sym
^2f)$.

\subsection{Asymptotics of the quantum variance }
For a set $A$ we let $1_A$ denote the indicator function of that set. Let $C_{0,0}^\infty(M, B)$ denote functions in $C_{0}^\infty(M, B)$ that are orthogonal to the constant function. Let $C_{\rm cusp}^\infty(M, B)$ be the subset of functions with zero-th Fourier coefficient vanishing completely, and $C_{\rm Eis}^\infty(M, B)$ its orthogonal complement inside $C_{0}^\infty(M,B)$.  We note that for $\psi\in C_{0,0}^\infty(M, B)$ we have $\nu(M_{(k-1)^\theta}\psi)=0$. If we restrict to test functions in this space  we can improve on Theorem \ref{ergodicity1} and  obtain an asymptotic result.

Denote by $\tau_1(n)$  the sum of divisors of $n$, $K_s(y)$ the $K$-Bessel function, and $\lambda_\phi(n)$ the $n$-th Hecke eigenvalue for the form $\phi$.

\begin{thm}\label{intro:QuantumVariance}
Let $0<\theta<1$ and fix $u:\R_+\to \R_{\geq 0}$ smooth with compact support.
\begin{enumerate}[label=(\roman*)]
    \item \label{ener}There exists a Hermitian form $B_\theta(\cdot{,}\cdot)$ on $C_{0,0}^\infty(M, B)$ and $\delta_\theta>0$ such that
  \begin{align}
  \sum_{2\vert k }u\left(\frac{k-1}{K}\right)\sum_{f\in H_k}&L(1, \sym^2{f})\abs{\mu_f(M_{(k-1)^{\theta}} \psi)}^2\\
  &= B_\theta(\psi,\psi) \left(\int u(y)y^{-\theta}dy\right) K^{1-\theta}+O_{\psi,\theta}(K^{1-\theta-\delta_\theta}),
\end{align}
for $\psi \in C_{0,0}^\infty(M, B)$.
\item \label{splitting-form}The Hermitian forms $B_\theta(\cdot{,}\cdot)$ have three different regimes in the sense that $B_\theta(\cdot{,}\cdot)$ is constant on each of the three intervals $0<\theta<1/2$, $\theta=1/2$ and $1/2<\theta<1$.\\
The decomposition \begin{equation}C_{0,0}^\infty(M, B)=C_{\rm cusp}^\infty(M, B)\oplus C_{\rm Eis}^\infty(M, B),\end{equation}
into the cuspidal and the Eisenstein part is orthogonal with respect to $B_\theta(\cdot{,}\cdot)$ for all $0<\theta<1$. Furthermore, $B_\theta(\cdot{,}\cdot)$ restricted to $C_{\rm Eis}^\infty(M, B)$ is independent of $\theta$, and $B_\theta(\cdot{,}\cdot)$ is identically zero on $C_{\rm cusp}^\infty(M, B)$ for $\theta>1/2$.

\item \label{extension}The Hermitian forms $B_\theta(\cdot{,}\cdot)$ can be extended to the larger set $1_BC_{0,0}^\infty(M)$ of functions in $C_{0,0}^\infty(M)$ times the characteristic function $1_B$ such that the following holds: On the subset $1_BC_{\rm cusp}^\infty(M)$ of functions with the zero-th Fourier coefficient vanishing, the form $B_\theta(\cdot{,}\cdot)$ is continuous with respect to a certain Sobolev norm $\norm{\cdot}_{2,1}$. The set $C_{\rm cusp}^\infty(M,B)$ is dense in $1_BC_{\rm cusp}^\infty(M)$ with respect to the same norm $\norm{\cdot}_{2,1}$.
\item \label{finishline} If
$\phi_i$ are Hecke--Maass forms with eigenvalue $s_i(1-s_i)$, then the Hermitian form satisfies $B_\theta(1_{B}\phi_1,1_{B}\phi_2)=0$, unless  $\phi_1,\phi_2$ are both even. If $\phi_i$ are both  even, then
\begin{align}
B_\theta(1_{B}\phi_1,1_{B}\phi_2)=4\pi \sum_{m,n\geq 1}  \frac{\tau_1((m,n))\lambda_{\phi_1}(m)\lambda_{\phi_2}(n)}{(mn)^{1/2}} I^{s_1,s_2}_\theta (m,n),\end{align}
where
\begin{equation}
    I_\theta^{s_1,s_2} (m,n)=\int_{\max(m,n)}^\infty K_{s_1-1/2}(2\pi y)\overline{K_{s_2-1/2}(2\pi y)} f_{\theta,m,n}(y)\frac{dy}{y}
\end{equation}
with
\begin{equation}
     f_{\theta,m,n}(y)=
     \begin{cases}
       1, &\textrm{ if }0<\theta<1/2,\\
        e^{-2\pi^2y^2(m^2+n^2)}, &\textrm{ if }\theta=1/2,\\
       0, &\textrm{ if }\theta>1/2.
     \end{cases}
 \end{equation}
\end{enumerate}
\end{thm}
For the precise form of $B_\theta(\cdot{,}\cdot)$ and $\norm{\cdot}_{2,1}$ we refer to  \eqref{Btheta} and \eqref{sobolev-again}.

Luo and Sarnak \cite[p.773]{LuoSarnak:2004a} proved that  $L(\phi,1/2)$  is  non-negative for  $\phi$  a Hecke--Maass cusp form by realizing it as an eigenvalue of the Hermitian form $B_0$. One may speculate whether  $B_\theta(1_{B}\phi,1_{B}\phi)$  for $0<\theta \leq 1/2$ is also related to central values of $L$-functions. Irrespectively, we may use Theorem \ref{intro:QuantumVariance} to prove that $B_\theta(1_{B}\phi,1_{B}\phi)\geq 0$.
 Seeing this directly from the series representation in Theorem \ref{intro:QuantumVariance} \ref{finishline} seems difficult, and is, therefore,  surprising.

In fact this was our original motivation for extending $B_{\theta}$ in Theorem \ref{intro:QuantumVariance} \ref{extension} to a set containing $1_B\phi$. Notice that $1_B\phi$ together with incomplete Eisenstein series provide a basis for $1_B C_{0,0}^\infty(M)$.

\begin{cor}\label{geq0} If $\phi$ is an even Hecke--Maass cusp form with eigenvalue $s_\phi(1-s_\phi)$ and Hecke eigenvalues $\lambda_\phi(n)$, then
\begin{align}
 \sum_{m,n\geq 1}  \frac{\tau_1((m,n))\lambda_\phi(m)\lambda_\phi(n)}{(mn)^{1/2}}\int_{\max(m,n)}^\infty \abs{K_{s_\phi-1/2}(2\pi y)}^2 \frac{dy}{y}\geq 0 .\end{align}
\end{cor}
\begin{rem}
Let $w:\R\to \R$ be a smooth and bounded weight function with support contained in $[1,\infty)$. Then one can similarly show by using the explicit expression for $B_\theta(\cdot{,}\cdot)$ in \eqref{Btheta} combined with Theorem \ref{intro:QuantumVariance} \ref{ener} for $\psi(z)=w(y)\phi(z)$ that
\begin{align}
 \sum_{m,n\geq 1}  \frac{\tau_1((m,n))\lambda_\phi(m)\lambda_\phi(n)}{(mn)^{1/2}}\int_0^\infty \abs{K_{s_\phi-1/2}(2\pi y)}^2w\left(\frac{y}{m}\right)w\left(\frac{y}{n}\right) \frac{dy}{y}\geq 0 .\end{align}
\end{rem}
\begin{rem}
 We expect that the techniques and results in this paper will work with some modifications also for Maass cusp forms in the same way that  the results in \cite{LuoSarnak:1995} are extended to the Maass case by Sarnak and Zhao \cite{SarnakZhao:2019}. For simplicity and clarity we restrict ourselves to the holomorphic case.
\end{rem}

\subsection{The behavior of holomorphic cusp forms high in the cusp}
Ghosh and Sarnak \cite{GhoshSarnak:2012} considered the distribution of the zeroes of holomorphic modular forms high in the cusp as the weight grows. By the work of Rudnick \cite{Rudnick:2005a} mass equidistribution for holomorphic forms implies equidistribution of their zeroes in the fundamental domain. Ghosh and Sarnak observed that, although the proportion of zeroes in a shrinking ball around the cusp (more precisely $H\gg \sqrt{k\log k}$) was proportional to the area of the domain, the statistical behavior of the zeroes was very different. They observed experimentally that the zeroes tend to localize on the two {\lq real\rq} lines $\Re z=-1/2$ and $\Re z=0$, conjectured that $100\%$ of the zeroes in these shrinking balls around the cusp should lie on these two lines, and obtained some results in this direction. These results were then strengthened by Lester, Matom\"aki, and Radziwi\l\l{}  \cite{LesterMatomakiRadziwil-l:2018a}.

The reason for the qualitative change in the behavior of holomorphic cusp forms high in the cusp has its roots in the fact that for all integers $1\ll l\ll \sqrt{k /\log k}$, we have
\begin{equation}\label{GhoshSarnak}\left(\frac{e}{l}\right)^{k-1}f(x+iy_l)=\lambda_f(l)e(xl)+O(k^{-\delta}),   \end{equation}
where $y_l=(k-1)/{4\pi l}$, and
 $\delta>0$ is some constant. This means that counting zeroes on the real lines reduces to detecting  sign-changes of the Hecke eigenvalues $\lambda_f(l)$, which is exactly what was achieved in \cite{LesterMatomakiRadziwil-l:2018a}.

We observe that our bilinear form $B_\theta(\cdot{,}\cdot)$ exhibits  a {\it phase transition} at $\theta=1/2$, which coincides exactly with the threshold in \cite{GhoshSarnak:2012} and \cite{LesterMatomakiRadziwil-l:2018a}. Combined, these results point towards the phenomenon that, although the mass of holomorphic cusp forms equidistribute all the way down to the Planck scale i.e. $k^{-1}$, the qualitative behavior changes high in the cusp at \emph{half} the Planck scale $k^{-1/2}$. This shows quite clearly in Figure \ref{fig1}, where the holomorphic forms look like random waves in the bottom of the plots, whereas in an intermediate range (at height around $k^{1/2}$ i.e. at half the Planck scale) they are essentially constant on horizontal lines, before they start decaying rapidly high in the cusp  at height around $k/(4\pi)$ i.e. below the Planck scale.
Note that in the region $\sqrt{k}\ll y\ll k$, where $y^k\abs{f(z)}^2$ is essentially constant on horizontal lines, we still expect fluctuations in the $y$ direction (as we expect mass equidistribution to hold all the way down to the Planck scale). In order to see this numerically one needs to consider larger $k$.

The asymptotic \eqref{GhoshSarnak} implies that $y^k|f(x+iy)|^2$ is essentially constant as $x$ varies, at least when $y=y_l$ for some $l$ as above. This provides intuition for the phenomena observed in this paper: $y^k|f(x+iy)|^2$ exhibits very strong cancellation with cuspidal test functions when we go to scales finer than halfway to the Planck scale. On the other hand for incomplete Eisenstein series the behavior is the same all the way down to the Planck scale, according to Theorem \ref{intro:QuantumVariance} \ref{splitting-form}.
\begin{figure}
    \centering
    \includegraphics[scale=1]{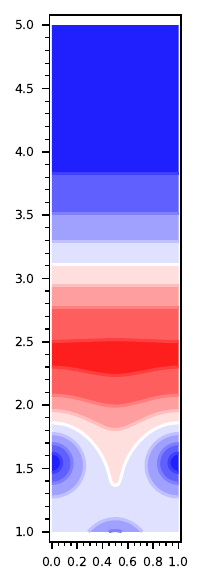}\includegraphics[scale=1]{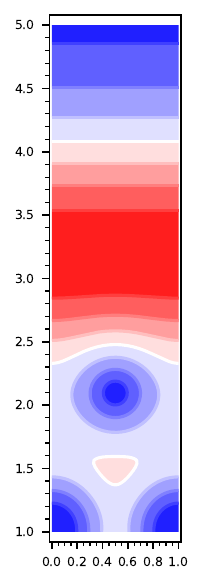}\includegraphics[scale=1]{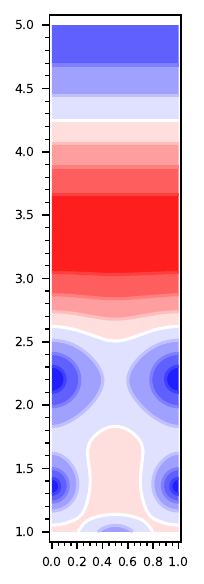}\includegraphics[scale=1]{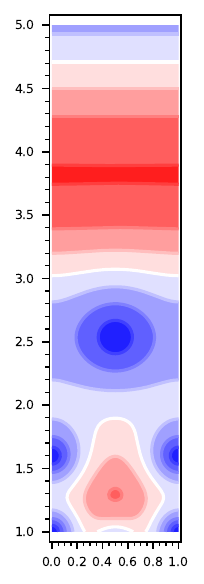}
    \caption{Heat plots of $y^k |f(z)|^2$ with $f\in \mathcal{S}_k(\Gamma_0(2))$ and $k=30,40,42,48$.}
    \label{fig1}
\end{figure}

The structure of the paper is as follows. In Section 2 we study the shifted convolution problem and its variance over a Hecke basis. In Section 3 we use the results of Section 2 to study the quantum variance when we squeeze non-holomorphic Poincar\'e series towards the cusp . In Section 4 we extend the space of observables to the space $C_{0,0}^\infty (M,B)$. In Section 5 we show that quantum ergodicity holds for shrinking sets towards the cusp down to the Planck scale. In Section 6
we complete the proof of Theorem \ref{intro:QuantumVariance} parts \ref{extension} and \ref{finishline}.

\section*{Acknowledgements}
We are grateful to the anonymous referees for their many insightful comments and useful suggestions.

  \section{The variance of shifted convolution sums over a Hecke basis}
An essential tool in understanding questions of equidistribution of Hecke eigenforms is understanding shifted convolution sums.

 Let $f$ be a weight $k$, level one holomorphic cuspidal Hecke eigenform, normalized such that its Fourier expansion \begin{equation}
f(z)=\sum_{n=1}^\infty\lambda_f(n)n^{\frac{k-1}{2}}e(nz)
\end{equation}  satisfies $\lambda_f(1)=1$. As usual,  $e(z)=e^{2\pi i z}$. The normalized Hecke eigenvalues satisfy the Hecke relations
 \begin{equation}\label{hecke-relations}
     \lambda_f(n)\lambda_f(m)=\sum_{d\vert(m,n)}\lambda_{f}\left(\frac{mn}{d^2}\right),
 \end{equation}
 see \cite[(6.38)]{Iwaniec:1997a}.
  Consider the shifted convolution sum
\begin{align}
      A_f^{W}(X,h)& :=\sum_{n\in \N}\lambda_f(n)\lambda_f(n+h) W((n+h/2)/X)\\
\label{after-Hecke-relations}      &=\sum_{d\mid h} \sum_{r\in \N}\lambda_f(r(r+d))W\left(\frac{\frac{h}{d}(r+d/2)}{X}\right),
  \end{align}
where $W:\R_+\to\R$ is smooth and supported in a compact interval, and where in the second line we have used the Hecke relations \eqref{hecke-relations}.

Let $\tau_1(n)=\sum_{d\mid n} d$, and let $L(s, \sym^2 f)$ be  the symmetric square $L$-function associated to $f$, i.e.
\begin{equation}
   L(s, \sym^2 f)=\zeta(2s)\sum_{n=1}^\infty\frac{\lambda_f(n^2)}{n^s}, \quad\textrm{ when }\Re(s)>1,
\end{equation}and is defined on $\C$ by analytic continuation.

  We investigate the variance of the smooth shifted convolution sums $A_f^{W}(X,h)$ over an orthonormal basis of Hecke eigenforms $H_k$ and over $k$ of size $K$. Let $u:\R_+\to\R_{\geq 0}$ be a compactly supported function. We want to understand \begin{equation}\label{understand-this}
\sum_{2\vert k}u\left(\frac{k-1}{K}\right)\frac{2\pi ^2}{k-1}\sum_{f\in H_k} \frac{A_f^{W_1}(X(k),h_1)\overline{A_f^{W_2}(X(k),h_2)}}{L(1, \sym^2 f)},
\end{equation}
where \begin{equation}X(k)=(k-1)^{1-\theta}\end{equation} for some $0<\theta<1$.

 In order to  describe better the dependence on $W, h$ we use Sobolev norms
 \begin{align}
 \begin{split}
 \label{sobolev-norms}\norm{W}_{l,p}^p=&\sum_{0\leq i\leq l}\norm{\frac{d^i}{dy^i}W}_p^p,\\ \norm{W}_{l,\infty}=&\sum_{0\leq i\leq l}\norm{\frac{d^i}{dy^i}W}_\infty.
 \end{split}
 \end{align}

For all compactly supported functions $W$ we choose $a_W>0, A_W>1$ such that $\supp{W}\subseteq [a_W,A_W]$. For $h_1$, $h_2\geq 1$ we denote $\norm{h}_\infty=\max(h_1,h_2)$.

 The main tool in understanding \eqref{understand-this} is the Petersson formula, which states that
  \begin{align} \frac{2\pi ^2}{k-1}& \sum_{f\in H_k} \frac{\lambda_f(n_1)\lambda_f(n_2)}{L(1, \sym^2 f)}\\& = \delta_{n_1,n_2}+2\pi (-1)^{k/2} \sum_{\substack{c\geq 1}} \frac{S(n_1,n_2;c)}{c}J_{k-1}\left( \frac{4\pi \sqrt{n_1n_2}}{c}\right), \label{Petersson}\end{align}
   see e.g. \cite[p. 776]{LuoSarnak:2004a}.
   We will use  the following estimate for the   $J$-Bessel function:
    \begin{equation} \label{good-bound-Bessel} J_{k-1}(x)\ll \left( \frac{ex}{2k}\right)^{k-1}, \quad x>0,\end{equation}
  see e.g. \cite[p. 233]{LiuMasri:2014}.

  To state our theorem we define, for functions $W_1, W_2:\R_+\to \R$ and $h_1,h_2\in \N$,
 \begin{align}\label{BLF} B_{h_1,h_2}(W_1, W_2)=  \tau_1((h_1,h_2))\int_{0}^\infty W_1(h_1 y) \overline{W_2(h_2y)}dy. \end{align}

  We now  prove the following result.
   \begin{thm}\label{QVthm}
   Let $0\leq \theta <1$.
  Let $u:\R_+\to \R_{\geq 0}$ be a smooth compactly supported weight function, and let $W_1, W_2:\R_+\to \R$ be smooth functions compactly supported below $A_{W_i}\geq 1$. Then
  \begin{align}
  \label{SCSsmooth}&\sum_{2\vert k}u\left(\frac{k-1}{K}\right)\frac{2\pi ^2}{k-1}\sum_{f\in H_k} \frac{A_f^{W_1}(X(k),h_1)\overline{A_f^{W_2}(X(k),h_2)}}{L(1, \sym^2 f)}\\
  = &B_{h_1,h_2}(W_1,W_2)\frac{K^{2-\theta}}{2}  \int_{0}^\infty u(y)y^{1-\theta}dy+ O_{W_i,h_i, \theta}(K).
  \end{align}
The implied constant in the error term may be bounded by a constant depending only on $\theta$ times
\begin{equation}(1+\norm{h}_\infty)^{1+\e}(A_{W_1}A_{W_2})^C \norm{W_1}_{C,\infty}\norm{W_2}_{C,\infty}\end{equation} for $C$ sufficiently large depending on $\theta$.
\end{thm}
  \begin{proof}
    Using \eqref{after-Hecke-relations} and the Petersson formula \eqref{Petersson}  we find,
    for all $X\geq 1$,
  \begin{align}
\label{some-expression}  &\frac{2\pi ^2}{k-1}\sum_{f\in H_k} \frac{A_f^{W_1}(X,h_1)\overline{A_f^{W_2}(X,h_2)}}{L(1, \sym^2 f)}
   \\
  &=\sum_{\substack{d_1 \mid h_1 \\ d_2 \mid h_2}} \sum_{r_1,r_2\in\N} \delta_{r_1(r_1+d_1)=r_2(r_2+d_2)} W_1\left( \frac{h_1(r_1+d_1/2)}{d_1 X} \right)\overline{W_2\left( \frac{h_2(r_2+d_2/2)}{d_2X} \right)}\\
  +&2\pi (-1)^{k/2} \sum_{\substack{d_1 \mid  h_1 \\ d_2 \mid h_2}} \sum_{r_1,r_2\in \N} W_1\left( \frac{h_1(r_1+d_1/2)}{d_1 X} \right)\overline{W_2\left( \frac{h_2(r_2+d_2/2)}{d_2X} \right)}\\
  &\times \sum_{\substack{c\geq 1}} \frac{S(r_1(r_1+d_1),r_2(r_2+d_2);c)}{c}J_{k-1}\left( \frac{4\pi \sqrt{r_1(r_1+d_1)r_2(r_2+d_2)}}{c}\right).
  \end{align}
We refer to the line with the Kronecker delta as the \emph{diagonal term}, and the rest as the \emph{off-diagonal term}.

  To handle the diagonal term, we observe that for fixed positive  $d_1\neq d_2$ the equation
  \begin{equation}\label{finitely-many}r_1(r_1+d_1)=r_2(r_2+d_2)\end{equation}
  has only finitely many positive solutions. To see this we rewrite \eqref{finitely-many} as
  \begin{equation}\label{finitely-many-2}(2r_1+d_1)^2-(2r_2+d_2)^2=d_1^2-d_2^2.\end{equation}
  Factoring the left-hand side as $(2r_1+d_1+2r_2+d_2)(2r_1+d_1-2r_2-d_2)$
  we see that any solution gives a factorization of  $d_1^2-d_2^2$, and that any factorization of $d_1^2-d_2^2$ comes from at most one solution.
  This shows that there are at most $d(d_1^2-d_2^2)$  solutions to  with $d_1\neq d_2$, where $d(n)$ denotes the number of divisors of $n$; indeed we see that the total contribution from these terms is $O(\norm{h}_\infty^{ \e}\norm{W_1}_\infty\norm{W_2}_\infty)$.

  For the remaining terms, i.e. $d_1=d_2=d$, $r_1=r_2=r$  we apply first Poisson summation in the $r$-variable  and observe that the Fourier transform $g\mapsto \hat g(t)=\int_{\R}g(x)e(tx)dx$ of the function $y\mapsto W_1\left( \frac{h_1(y+d_1/2)}{d_1X} \right)\overline{W_2\left( \frac{h_2(y+
  d_2/2)}{d_2X}\right)}$ at $r$ is bounded by an absolute constant times
  \begin{equation} \label{coffeetime}
      \abs{2\pi r}^{-n}(dX)^{-n+1}\norm{W_1(h_1\cdot)W_2(h_2\cdot)}_{n,1},\end{equation}
 which follows from repeated integration by parts. We now see that
 \begin{equation}
     \sum_{\substack{d_1 \mid h_1 \\ d_2 \mid h_2\\ d_1=d_2}} \sum_{r\in\N} W_1\left( \frac{h_1(r+d_1/2)}{d_1X} \right)\overline{W_2\left( \frac{h_2(r+
  d_2/2)}{d_2X}\right)}
 \end{equation}
 equals the same expression with the sum over $r\in \N$ replaced by the same sum over $r\in \Z$ up to an error term of $O(\norm{h}_\infty
 ^{1+\e}\norm{W_1}_\infty \norm{W_2}_\infty)$.
We then observe that
  \begin{align}
 \label{extendtoZ} \sum_{d\vert(h_1,h_2)} &  \sum_{r\in\Z} W_1\left( \frac{h_1(r+d/2)}{dX} \right)\overline{W_2\left( \frac{h_2(r+
  d/2)}{dX}\right)}\\
  &= \sum_{d\mid (h_1,h_2)} \sum_{r\in \Z}\int_{-\infty}^\infty W_1\left(\frac{h_1 (y+d/2)}{dX}\right)\overline{W_2\left(\frac{h_2 (y+d/2)}{dX}\right)}e(-ry)dy\\
  &= \tau_1((h_1,h_2))\int_{-\infty}^\infty W_1(h_1 y) \overline{W_2(h_2y)}dyX +O(\norm{h}_\infty^{\e}\norm{W_1(h_1\cdot)W_2(h_2\cdot)}_{2,1}),
  \end{align}
where we have extended trivially the $r$-sum to all of $r$, then used Poisson summation and the bound \eqref{coffeetime} with $n=2$. Now we average over $k$ and apply Poisson summation in the $k$-variable. Using integration by parts on the dual side we find that for any $A>0$, we have
  \begin{align*}\sum_{2\vert k} u\left(\frac{k-1}{K}\right)X(k)&= \frac{K^{2-\theta}}{2}\int_0^\infty u(y)y^{1-\theta}dy+O_{u,A}(K^{-A}),\\
  \sum_{2\vert k} u\left(\frac{k-1}{K}\right)&= \frac{K}{2}\int_0^\infty u(y)dy+O_{u,A}(K^{-A}),
  \end{align*}
  which yields the desired main term up to the stated error term.

  For the off-diagonal terms we need to bound
  \begin{align}
 \sum_{2\vert k}u\left(\frac{k-1}{K}\right) 2\pi (-1)^{k/2}& \sum_{\substack{d_1 \mid  h_1 \\ d_2 \mid h_2\\r_1,r_2\in \N}} W_1\left( \frac{h_1(r_1+d_1/2)}{d_1 (k-1)^{1-\theta}} \right)\overline{W_2\left( \frac{h_2(r_2+d_2/2)}{d_2(k-1)^{1-\theta}} \right)}\\
  &\times \sum_{\substack{c\geq 1}} \frac{S(r_1(r_1+d_1),r_2(r_2+d_2);c)}{c}J_{k-1}\left( \frac{\Delta}{c}\right),
  \end{align}
  where $\Delta=4\pi \sqrt{r_1(r_1+d_1)r_2(r_2+d_2)}.$
We mimic the arguments of Luo and Sarnak \cite[p 880--881]{LuoSarnak:2003a}. We start by noticing that
\begin{enumerate}[label=(\roman*)]
    \item\label{one} the summation over $k$  is supported in  $K\ll_u k\ll_u K$,
\item \label{two} the summations over  $r_i$  are supported in
\begin{align}
a_{W_i}\frac{d_i}{h_i}(k-1)^{1-\theta}&\leq (r_i+d_i/2)\leq A_{W_i}\frac{d_i}{h_i}(k-1)^{1-\theta}\left(\ll_u A_{W_i}\frac{d_i}{h_i} K^{1-\theta}\right).
\end{align}
\end{enumerate}

Using again $r_i(r_i+d_i)=(r_i+d_i/2)^2-d_i^2/4$  we see that
\begin{enumerate}[label=(\roman*), resume]
    \item\label{three} in the support of the above  sums, we have \begin{equation}
    \Delta\ll_{u} \norm{h}^\e_\infty A_{W_1}A_{W_2}\frac{d_1d_2}{h_1h_2}K^{2(1-\theta)}.\end{equation}
\end{enumerate}

We want to truncate the sum over $c$ and notice that for $r_1,r_2$ in the support of the sums we may use the bound \eqref{good-bound-Bessel} on the Bessel function and the trivial bound on the Kloosterman sum to get
\begin{align}
    \sum_{\substack{c\geq M}}& \frac{S(r_1(r_1+d_1),r_2(r_2+d_2);c)}{c}J_{k-1}\left( \frac{\Delta}{c}\right)\\
    \ll& \sum_{c\geq M}\left(\frac{ C_uA_{W_1}A_{W_2}d_1d_2K^{1-2\theta}}{h_1h_2c}\right)^{k-1}\ll_u \left(\frac{C_uA_{W_1}A_{W_2}d_1d_2K^{1-2\theta}}{h_1h_2M}\right)^{k-1}\frac{M}{K}.
\end{align}
   We conclude that, if $M=C_uA_{W_1}{A_{W_2}}\frac{d_1d_2}{h_1h_2}K^{1-2\theta+\e}$, this term decays exponentially in $K$. Therefore
 \begin{enumerate}[label=(\roman*), resume]
    \item \label{four} the sum in $c$ above may be truncated at \begin{equation}c\ll_u A_{W_1}A_{W_2} \frac{d_1d_2}{h_1h_2}K^{1-2\theta+\e}\end{equation} \end{enumerate}
up to  an additional  error of
    $\ll_{u} \norm{W_1}_\infty\norm{W_2}_\infty A_{W_1} A_{W_2}K^{-A}$.
  We now quote lemmata 4.1 and 4.2 in \cite{LuoSarnak:2003a} stating that
 for $g$  a smooth function compactly supported  on $\R_+$
 we have
  \begin{align}
  \label{first}  \quad  \sum_{2\vert k} 2\pi (-1)^{k/2} J_{k-1}(x)g(k-1)&=-2\pi \int_{-\infty}^\infty \hat{g}(t) \sin(x\cos(2\pi t))dt,\\
 \label{second}\int_{-\infty}^\infty \hat{g}(t) \sin(x(1-2\pi^2t^2))dt&= \int_0^\infty  \frac{g(\sqrt{2yx})}{(\pi y)^{1/2}}\sin(y+x-{\pi}/{4})dy,\\
\label{third}  \int_{-\infty}^\infty \hat{g}(t) \cos(x(1-2\pi^2t^2))dt&= \int_0^\infty \frac{g(\sqrt{2yx})}{(\pi y)^{1/2}} \cos(y+x-{\pi}/{4})dy.
\end{align}
 In our case we apply \eqref{first} to the function
  \begin{equation}g(y)= u(y/K) W_1\left(\frac{h_1(r_1+d_1/2)}{d_1y^{1-\theta}}\right) \overline{W_2}\left(\frac{h_2(r_2+d_2/2)}{d_2y^{1-\theta}}\right) .\end{equation}
  This shows that the remaining part of the non-diagonal contribution can be bounded by an absolute constant times
  \begin{align}
\label{first-transform}  \sum_{\substack{d_i \mid h_i }} \sum_{r_i\geq 1} \sum_{\substack{c\geq 1}} \abs{\int_{-\infty}^\infty \hat{g}(t) \sin \left( \frac{\Delta}{c}  \cos(2\pi t)\right)dt},
  \end{align}
  with restrictions on the sums as \ref{two}-\ref{four} above. Here we have used the trivial estimate on the Kloosterman sums.

  As in \cite[Eq (4.4)]{LuoSarnak:2003a} we now use a trigonometric identity and Taylor expansions  to get for $x, t\in \R$
  \begin{align}
   \sin(x\cos(2\pi t))&=\sin\left(x(1-2\pi^2t^2)+x\sum_{n\geq 2}(-1)^n\frac{(2\pi t)^n}{(2n)!}\right)\\
  \label{taylorsin} &= \sin(x(1-2\pi^2t^2))\left( \sum_{0\leq n,m\leq N-1} c_{m,n} (xt^4)^{2n}t^{2m}  \right)\\
  &\quad +\cos(x(1-2\pi^2t^2))\left( \sum_{\substack{1\leq n\leq N
  \\0\leq m\leq N-1}} d_{m,n} (xt^4)^{2n-1}t^{2m}  \right)\\
  &\quad +O((xt^4)^{2N}+(xt^4)^{4N}+t^{2N}+t^{4N}),
  \end{align}
  for any $N$, where $c_{m,n}$ and $d_{m,n}$ are real constants. In order to bound the term coming from the error term above, we observe that all derivatives $g^{(m)}$ are supported in $K\cdot \supp(u)$ and we claim that, when $r_1,r_2$ satisfy \ref{one}--\ref{four}, we have the bound
  \begin{equation} \label{boundderivative} g^{(m)}(y)\ll_{u,m}C_{W_1,W_2,m} K^{-m}, \end{equation}
  where $C_{W_1,W_2,m}=\prod_{i=1,2}\norm{W_i}_{m,\infty}A_{W_i}^m$.

  To see why the claim is true we observe from Leibniz' rule
that  $g^{(m)}(y)$ is bounded by an absolute constant (depending on $m$) times  \begin{align*}
   \max_{p_1+p_2+p_3=m} \abs{ \frac{d^{p_1}}{d y^{p_1}}u(y/K) \frac{d^{p_2}}{d y^{p_2}}  W_1\left(\frac{h_1(r_1+d_1/2)}{d_1y^{1-\theta}} \right) \frac{d^{p_3}}{d y^{p_3}}\overline{W_2}\left(\frac{h_2(r_2+d_2/2)}{d_2y^{1-\theta}}
 \right)}.
  \end{align*}

 Now we observe that by the chain rule
  \begin{equation}\frac{d^{p}}{d y^{p}}u(y/K)= u^{(p)}(y/K) K^{-p}\ll_{u,p} K^{-p}. \end{equation}
Using Fa\`a di Bruno's formula for the higher derivative we see that
  \begin{align}
  \frac{d^{p}}{d y^{p}} W_i\left(\frac{h_i(r_i+d_i/2)}{d_iy^{1-\theta}}\right)
 &\ll_p\norm{W_i}_{p,\infty}\sum \prod_{j=1}^p\left(\frac{h_i(r_i+d_i/2)}{d_iy^{-\theta+1+j} }\right)^{m_j}\\
 &=\norm{W_i}_{p,\infty}\sum \left(\frac{h_i(r_i+d_i/2)}{d_iy^{-\theta+1} }\right)^{\sum_{i}m_i}y^{-p}.
 \end{align}
 The sum is over $p$-tuples of integers satisfying $m_1+2 m_2+\cdots pm_p =p$. Using that $W_i$ is supported in $[a_{W_i},A_{W_i}]$ we see that we may bound the term inside the   parentheses in the last equation by $A_{W_i}$. For $y$ in the support of $g$ we have $y\in K\cdot \supp{u}$ so for such $y$ we get
  \begin{equation}\frac{d^{p}}{d y^{p}} W_i\left(\frac{h_i(r_i+d_i/2)}{d_iy^{1-\theta}}\right) \ll_{u,p}{\norm{W_i}}_{p,\infty}A_{W_i}^p K^{-p}.  \end{equation}
  Combining these bounds proves the claim \eqref{boundderivative}.

  From  \eqref{boundderivative} it follows that $\widehat{g^{(m)}}(y)\ll_{u,m}C_{W_1,W_2,m} K^{-(m-1)}$. Additionally partial integration gives  $\widehat{g^{(m)}}(t)\ll_{u,m} C_{W_1,W_2,m+l}\abs{t}^{-l}K^{-(m+l-1)} $ so by using $\widehat{g^{(m)}}(y)= (-2\pi i y)^m \hat{g}(y)$ we may conclude, by using the first bound for $\abs{t}\leq K^{-1}$ and the second bound with $l=2$ when $\abs{t}>K^{-1}$  that
  \begin{equation}\label{ref:zombie}
  \int_{-\infty}^\infty \abs{\hat{g}(t) t^m} dt \ll_{u,m} C_{W_1,W_2,m+2}K^{-m}.
  \end{equation}

 Using this bound we see that, when we use  the Taylor expansion \eqref{taylorsin}, the contribution from
 \begin{enumerate}
     \item $((\Delta/c)t^4)^{2N}$ is $\ll_u \norm{h}_\infty^\e C_{W_1,W_2,8N+2}\prod_{i=1,2}A_{W_i}^{2N+1}K^{(1-\theta)2(2N+1)-8N}$,
     \item $((\Delta/c)t^4)^{4N}$ is $\ll_u\norm{h}_\infty^\e C_{W_1,W_2,16N+ 2}\prod_{i=1,2}A_{W_i}^{4N+1}K^{(1-\theta)2(4N+1)-16N}$,
     \item $t^{2N}$ is $\ll_u \norm{h}_\infty^\e C_{W_1,W_2, 2N+2}\prod_{i=1,2}A_{W_i}^{2}K^{(3-4\theta)-2N}$,
     \item $t^{4N}$ is  $\ll_u \norm{h}_\infty^\e C_{W_1,W_2, 4N+2}\prod_{i=1,2}A_{W_i}^{2}K^{(3-4\theta)-4N}.$
 \end{enumerate}
 We note that for $N=1$ all terms are $\ll_u \norm{h}_\infty^\e C_{W_1,W_2,18}\prod_{i=1,2}A_{W_i}^5  K$.

To bound the remaining terms involving \begin{equation}\label{ref:star}e(x,t):=\sin(x(1-2\pi^2t^2))c_{00}+\cos(x(1-2\pi^2t^2))d_{01}xt^4\end{equation} coming from the Taylor expansion, we combine $\widehat{g^{(m)}}(y)= (-2\pi i y)^m \hat{g}(y)$ with  \eqref{third} which gives the bound
  \begin{equation}\label{it-never-ends}
\int_{-\infty}^\infty\hat g(t)e({\Delta}/{c},t)dt  \ll \max_{\substack{(m,m')=(4,1), (0,0)\\\pm}} \abs{\left(\frac{\Delta}{c}\right)^{m'} \int_{0}^\infty g^{(m)}\left(\sqrt{\frac{2\Delta}{c}y}\right)y^{-1/2} e^{\pm iy} dy},
  \end{equation}
  where we used Euler's formulas for sine and cosine. Now we apply partial integration to the integral with $e^{iy}$ as one of the functions.

  For $r_1,r_2, \Delta, c$ as in \ref{two}--\ref{four},i.e. where the terms in the sum \eqref{first-transform} might be non-vanishing, we claim that for any $n,m\in \mathbb{Z}_{\geq 0}$
\begin{equation}\label{ref:starstar}
\frac{d^n}{dy^n}\left(g^{(m)}\left(\sqrt{\frac{2\Delta}{c}y}\right) y^{-1/2}\right)\ll_{u ,n,m} \frac{C_{W_1, W_2, m+n}}{K^m y^{1/2+n}}.
\end{equation}

To see this we note that the left-hand side is non-zero only if $\Delta y/c\asymp_u K^2$.
By using the Leibniz rule and Fa\`a di Bruno's formula we see that
\begin{align}
\frac{d^n}{dy^n}&\left(g^{(m)}\left(\sqrt{\frac{2\Delta y}{c} }\right) y^{-1/2}\right)\ll_n \sum_{i=0}^n\abs{\frac{d^i}{dy^i}\left(g^{(m)}\left(\sqrt{\frac{2\Delta y}{c}}\right)\right)y^{-1/2-(n-i)}}\\
&\ll_{u,n}\sum_{i=0}^n\abs{\sum_{m_1,\ldots, m_i}\left(g^{(m+m_1+\ldots m_i)}\left(\sqrt{\frac{2\Delta y}{c}}\right)\right)\prod_{j=1}^i \left(\sqrt{\frac{2\Delta y}{c}}y^{-j}\right)^{m_j}}y^{-1/2-(n-i)}\\
&\ll_{u,n} C_{W_1, W_2, m+n}K^{-m} y^{-1/2-n}.
\end{align}
Here the inner sum is over $m_1, \ldots, m_i$ satisfying $m_1+2m_2+\ldots+im_i=i$ and in the last line we have used \eqref{boundderivative} and that $\Delta y/c\asymp_u K^2$.

  For $(m,m')=(4,1)$ in \eqref{it-never-ends} we use the claim with $n=0$ and for $(m,m')=(0,0)$ we take a general  $n$ which will eventually depend on $\theta$, and we find, by using integration by parts as described above,
 \begin{equation}
     \int_{-\infty}^\infty\hat g(t)e({\Delta}/{c},t)dt \ll_u
     C_{W_1,W_2,4}\frac{K^{-3}\Delta^{1/2}}{c^{1/2}}+C_{W_1,W_2,n}\frac{K^{1-2n}\Delta^{n-1/2}}{c^{n-1/2}}.
 \end{equation}
  Plugging this bound back in the sum \eqref{first-transform} and using the restriction \ref{two}-\ref{four} gives the result by choosing $n$ sufficiently large depending on $\theta$.
  \end{proof}

\begin{rem}
Note the resemblance between Theorem \ref{QVthm} and \cite[Thm 1.3]{NordentoftPetridisRisager:2020}.
 Whereas \cite[Thm 1.3]{NordentoftPetridisRisager:2020} is restricted to a range where the contribution of the individual off-diagonals  are essentially trivial due to the decay of the $J$-Bessel function (corresponding to $\theta>1/2$), we note that for $\theta \leq 1/2$ we need to exploit additional cancellation between the $J_{k-1}$-Bessel functions  for different $k$.
\end{rem}

  \section{Computing the quantum variance}
We now explain how the above results may be used to understand quantum variance for shrinking sets around the cusp.

\subsection{Squeezing sets towards cusps}\label{squeezing-operator}
Let $M=\GmodH$ be a finite volume hyperbolic surface. Then $M$ admits a decomposition
\begin{equation}
    M=M_0\cup Z_1\cup\ldots\cup Z_l,
\end{equation} where $M_0$ is compact and $Z_i$ is isometric to \begin{equation}
    Z_i\simeq S^1\times ]a_i,\infty[,
\end{equation}  for some $a_i>0$ with the metric on $S^1\times ]a_i,\infty[ $ equal to
\begin{equation}
    ds^2=\frac{dx^2+dy^2}{y^2}
\end{equation}
for $(x,y)\in S^1\times ]a_i,\infty[$. In the literature the regions $Z_i$ are called
horoball cusp neighbourhoods,
horocusps,
cuspidal zones,
Siegel sets,
horocyclic regions,
or simply (by an abuse of notation) cusps.
These subregions $Z_i$ are unbounded regions with boundary  the horocycle $(S^1\times \{a_i\})$ and a point (the cusp).

 We may assume that $Z=Z_1$ corresponds to a cusp at infinity. We now consider a measurable set $B\subseteq Z$  of hyperbolic volume $\vol{B}>0$ and define, for every $H\geq \vol{B}^{-1}$  the injective map

\begin{equation}
    \begin{tikzcd}[row sep=0pt,column sep=1pc]
 S_H^B\colon B \arrow{r} & Z \\
  {\hphantom{f\colon{}}} x+iy \arrow[mapsto]{r} & x+i\vol{B}H y,
\end{tikzcd}
\end{equation}
pushing  the region $B$ up towards the cusp at infinity.  We note that this may be formulated as a scaling along a geodesic going to the cusp thereby defining $S_H^B$ in a coordinate-free way.
We let $B_H=S_H^B(B)$ and notice that by a simple change of variables
\begin{align}
    \vol{B_H}&=\int_{a_1}^\infty\int_0^1 1_{B}((S_H^B)^{-1}z)\frac{dxdy}{y^2}\\
    &= \frac{1}{\vol{B}H} \int_0^\infty\int_{0}^{1}1_{B}(z)\frac{dxdy}{y^2}=\frac{1}{H}.
\end{align}
For $A\subseteq M$ we let
\begin{equation}
    L^2(M, A)=\{ f\in L^2(M) :  \supp{f}\subseteq A \}
\end{equation}
and define \emph{the squeezing operator}
\begin{equation}
    \begin{tikzcd}[row sep=0pt,column sep=1pc]
 M_H^B\colon L^2(M, B) \arrow{r} & L^2(M, B_H) \\
  {\hphantom{M_H^B\colon{}}} f \arrow[mapsto]{r}&{f\circ(S_H^B)^{-1}}
\end{tikzcd}
\end{equation}
i.e. $M_H^Bf(z)=f(x+iy/(\vol{B}H))$. We note that $M_H^B$ loosely speaking squeezes the function $f$ into the region $B_H$, which moves towards the cusp at infinity.

A simple change of variable computation -- similar to the volume computation of $\vol{B_H}$ above  -- shows that for $\varphi\in L^2(M, B)$
\begin{equation}
    \norm{M_H^B\varphi}^{2}=\frac{1}{\vol{B}H}\norm{\varphi}^2,\quad  \inprod{M_H^B\varphi}{1}=\frac{1}{\vol{B}H}\inprod{\varphi}{1}.
\end{equation}

We now specialize to $\G=\pslz$  and $Z=S^1\times ]1,\infty[$. For $T>1$ we let \begin{equation}B_T(\infty)=\{z \in Z : \Im(z)> T\},\end{equation} which we consider to be a ball around the cusp at infinity. A trivial computation shows that $\vol{B_{T}(\infty)}=1/T$. Fix now $T_0>1$ and let $B=B_{T_0}(\infty)\subseteq Z$.
With this choice of $B$ the squeezed set $B_H$ does not depend on $T_0$,
since we have $B_H=S_H^B(B)=B_H(\infty)$. Note, however, that the squeezing operator $M
^{B}_H$ still depends on the choice of $T_0$.

\subsection{Mass equidistribution in squeezed sets}
We now  consider the notion of mass equidistribution in the context of the squeezed sets as above:
Fix $H$, $\varphi\in L^{2}(M,B)$. It follows from the mass equidistribution theorem of Soundararajan and Holowinsky \cite{HolowinskySoundararajan:2010} that \begin{equation}
\int_{B_H}(M_H^B\varphi)(z) y^k\abs{f(z)}^2d\mu(z)=\frac{1}{\vol{M}}\int_{B_H}(M_H^B\varphi)(z) d\mu(z)+o_{M_H^B\varphi}(1),
\end{equation}
as $k\to\infty$.

We investigate what condition  on $H$ as a function of  $k$ implies that
\begin{equation}\label{equidistribution-shrinking-infinity}
\int_{B_H}(M_H^B\varphi)(z) y^k\abs{f(z)}^2d\mu(z)=\frac{1}{\vol{M}}\int_{B_H}(M_H^B\varphi)(z) d\mu(z)+o\left(\int_{B_H}\abs{M_H^B\varphi}d\mu\right)
\end{equation}
as $k, H\to\infty$.

Choosing $\varphi=1_B$ this simplifies to the question of when \begin{equation}\label{equidistribution-shrinking-infinity-special-case}
\int_{B_H}y^k\abs{f(z)}^2d\mu(z)=\frac{1}{H\vol{\GmodH}}+o(H^{-1}),
\end{equation}
as $k, H\to\infty$. However, we  investigate also more general test functions  $\varphi$.

 For the rest of the paper we fix $B=B_1(\infty)$, and consider the situation above for $H=(k-1)^\theta$ for some $\theta>0$, i.e. we consider
\begin{equation}M_{(k-1)^\theta}:=M^B_{(k-1)^\theta},\end{equation}
i.e. $M_{(k-1)^\theta}f(z)=f(x+iy/(k-1)^\theta).$
  We  investigate the mass equidistribution when the test function is squeezed via this operator  by considering the squeezing of the non-holomorphic Poincar\'{e} series
  \begin{equation}
      P_{\W,h}(z)= \sum_{\gamma \in \Gamma_\infty \backslash \Gamma} \W(y(\gamma z)) e(h x(\gamma z)),
  \end{equation}
 where $\W:\R_+\to\C$ is a smooth compactly supported function with support contained in $(1,\infty)$, and $x(z),y( z)$ are the real and imaginary parts of $z$.  In other words, we want to understand the asymptotic properties  of
  \begin{equation}\mu_f(M_{(k-1)^\theta}P_{\W, h}): = \langle M_{(k-1)^\theta}P_{\W,h}(z), y^k \abs{f(z)}^2 \rangle, \quad   k\to\infty.\end{equation}
We note that with our assumption on $\W$ the function $P_{\W,h}$ is supported on $B$, and that these  series actually span $L^2(M, B)$. In fact
\begin{equation}P_{\W,h}(z)=\W(y)e(hx) \quad \hbox{ for } z\in B. \end{equation}
For $h_1h_2\neq 0$ we define
\begin{equation}
    B_{\theta}(P_{\W_1,h_1},P_{\W_2,h_2})=\frac{\pi}{4} \tau_1((\abs{h_1},\abs{h_2}))\int_0^\infty \W_1(\frac{y}{\abs{h_1}}) \overline{\W_2(\frac{y}{\abs{h_2}})}f_{\theta,h_1,h_2}(y)\frac{dy}{y^2},
\end{equation}
where
 \begin{equation}
     f_{\theta,h_1,h_2}(y)=
     \begin{cases}
       1 &\textrm{ if }0<\theta<1/2,\\
        e^{-2\pi^2y^2(h_1^2+h_2^2)} &\textrm{ if }\theta=1/2,\\
       0 &\textrm{ if }\theta>1/2.
     \end{cases}
 \end{equation}

When $\theta =0$ we define $B_0$ to be the form $B_\omega$ defined by Luo and Sarnak in \cite[Eq (15)]{LuoSarnak:2004a}.
  \begin{thm}\label{thm:calculating-variances}
  Let $u:\R_+\to \R_{\geq 0}$ be a smooth compactly supported weight function, and let  $\W_1,\W_2$ be as above. For $h_1 h_2\neq 0$ and $0\leq \theta<1 $, we have
  \begin{align}
  \label{qvariance}\sum_{2\mid k} u\left(\frac{k-1}{K}\right)&
  \sum_{f\in H_k}L(1,\sym^2 f)  \mu_f(M_{(k-1)^\theta}P_{\W_1, h_1}) \overline{\mu_f(M_{(k-1)^\theta}P_{\W_2, h_2})}\\
   &= B_{\theta}(P_{\W_1,h_1},P_{\W_2,h_2}) \int_0^\infty u(y) y^{-\theta}dy\, K^{1-\theta} +O_{\theta,\eps, V_i,h_i}(K^{1-\theta-\delta_\theta+\e}),
  \end{align}
  where
  \begin{align}
  \delta_\theta  =  \begin{cases} (1+\theta)/2 & \theta\in (0,1/5), \\
  1-2\theta& \theta\in [1/5,1/2),\\
  1/2& \theta=1/2,\\
   1+2\theta& \theta\in (1/2,1).\end{cases}
\end{align}
  The implied constant in the error term may be bounded by a constant depending only on $\theta$, $\epsilon$ times
\begin{equation}(1+\norm{h}_\infty)^{C}(A_{V_1}A_{V_2})^C \norm{V_1}_{C,\infty}\norm{V_2}_{C,\infty},\end{equation} for $C$ sufficiently large depending on $\theta$.
\end{thm}
  \begin{proof}
By linearity we may assume $V_i$ to be real. For such a test function, say $V$,  we observe that $\mu_f(M_{(k-1)^\theta}P_{\W, h})$ is a real number (as is seen by unfolding and using that $\abs{f}^2$ is even). Therefore we have $\mu_f(M_{(k-1)^\theta} P_{\W,h}) =\mu_f(M_{(k-1)^\theta} P_{\W,-h}) $. So we may assume $h_i>0$ below.  We also notice that if $V$ is supported in $(1,A_V]$ then, up to an absolute constant times a power of $h$  times a power of $A_V$,  the functions $W_i(y)=V((4\pi y)^{-1})y^i$, $W^*(y)=V((4\pi y)^{-1})\exp(-h^2y^{-2}/8)$ all have Sobolev norms less than or equal to the corresponding Sobolev norm of $V$. This will be used below without further mention.

  The case $\theta=0$ is \cite[Thm 2]{LuoSarnak:2004a}. To handle the other cases we proceed as in the proof of \cite[Prop 2.1]{LuoSarnak:2003a}.  Doing this,  noticing in the proof that the Mellin transform satisfies
  \begin{equation}
      \int_{0}^\infty V\left(\frac{y^{-1}}{(k-1)^\theta}\right)y^{s}\frac{dy}{y}=(k-1)^{-s\theta}\int_{0}^\infty V\left(y^{-1}\right)y^{s}\frac{dy}{y},\end{equation}
 we find that
  \begin{align}\label{neverending} \mu_f(M_{(k-1)^\theta}P_{\W, h})
   = &\frac{2\pi^2}{(k-1)L(1,\sym^2 f)} \sum_{n\in \N} \lambda_f(n)\lambda_f(n+h) \\ &\times \W\left(\frac{(k-1)^{1-\theta}}{4\pi (n+h/2)}\right) \left(\frac{\sqrt{n(n+h)}}{n+h/2} \right)^{k-1}\\
& \quad\quad\quad\quad+O_{V,h}(k^{-1-\theta+\eps}),\end{align}
 where the implied constant is
 \begin{equation}\label{poly1}\ll_{\theta, \e}(1+h^B)A_{V}^B\norm{\W}_{B,\infty}\end{equation} for $B$ sufficiently large.
  This holds also for $h=0$.

  We now assume that $0<\theta<1/2$, and observe that in the above sum we may restrict to $n$ such that $(n+h/2)\asymp k^{1-\theta}$, which implies that $(k-1)/(n+h/2)^2=o(1)$ as $k\to \infty$. Therefore, we can employ the following Taylor expansion
  \begin{align}
    \left(\frac{\sqrt{n(n+h)}}{n+h/2} \right)^{k-1} &=\exp\left(\frac{k-1}{2}\log \left( 1-\frac{h^2}{(2n+h)^2} \right) \right)\\
    &=  \exp\left(-\frac{k-1}{2}\frac{h^2}{(2n+h)^2} +O\left( \frac{kh^4}{(2n+h)^4}\right) \right)\\
    &= \exp\left(-\frac{k-1}{2}\frac{h^2}{(2n+h)^2} \right)+O\left( \frac{kh^4}{(2n+h)^4}\right)\\
    &=  \sum_{i=0}^{N-1}\frac{\left(-\frac{h^2}{2}\frac{k-1}{(2n+h)^2}\right)^i}{i!}+O_N\left(\frac{(kh^2)^N}{(2n+h)^{2N}}+ \frac{kh^4}{(2n+h)^4}\right)\\
    &= \sum_{i=0}^{N-1}\frac{\left(-\frac{h^2}{8}\frac{k-1}{(n+h/2)^2}\right)^i}{i!}+O_N( h^{2N}k^{2N(\theta-1/2)}+h^{4}k^{4\theta-3}).
  \end{align}
  This gives us, using $L(1,\sym^2 f)\gg k^{-\e}$, see \cite[Eq. (2.1)]{LuoSarnak:2003a},
  \begin{align}
  \mu_f(M_{(k-1)^\theta}P_{\W, h})
  = &\frac{2\pi^2}{(k-1)L(1,\sym^2 f)} \sum_{i=0}^{N-1} \frac{\left(-\frac{h^2}{8}(k-1)^{2\theta-1}\right)^i}{i!} A^{W_i}_f((k-1)^{1-\theta},h)\\
  &\quad\quad + O_{V,h,N}((h^{2N}k^{2N(\theta-1/2)-\theta+\eps}+h^4k^{3\theta-3+\eps})+k^{-1-\theta+\eps}),\label{innerproduct-to-SCS}
  \end{align}
  where $W_i(y)= V\left(\frac{1}{4\pi y}\right)y^{-2i}$ for $y\in \R_+$, and the implied constant is of the form \eqref{poly1}. Since $\theta<1/2$, we can choose $N$ large enough such that the dominating error term in $k$  is $k^{-1-\theta+\eps}$.\\
 We now plug \eqref{innerproduct-to-SCS} into the expression we want to evaluate. The terms involving the products of error terms is easily seen to be $O_{V,h,N}(K^{-2\theta+\e})$.

 To bound the mixed terms we note that $(k-1)^{(2{\theta-1})i}$ is largest when $i=0$, so it suffices to observe that  \begin{align}
  K^{-1-\theta+\e}&\sum_{2|k}u\left(\frac{k-1}{K}\right) \frac{1}{k-1}\sum_{f\in H_k} \abs{A^{W_i}_f( (k-1)^{1-\theta},h)}\\
  &\ll_{V,h,N} K^{-1-\theta+\e} K^{1/2+\eps} \left( \sum_{k\geq 1,2|k} u\left(\frac{k-1}{K}\right) \sum_{f\in H_k} \frac{\abs{A^{W_i}_f( (k-1)^{1-\theta},h)}^2}{L(1, \sym^2 f)}  \right)^{1/2}\\
  &\ll_{V,h,N}   K^{1/2-3\theta/2+\eps},
  \end{align}
  where we have used the Cauchy--Schwarz inequality, the positivity of $L(1, \sym^2 f )$, and Theorem \ref{QVthm}. The implied constant is of the claimed form.
  This implies that
  \begin{align}
  \sum_{2|k} u&\left(\frac{k-1}{K}\right)\sum_{f\in H_k}L(1,\sym^2 f) \mu_f(M_{(k-1)^\theta}P_{\W_1, h_1})\overline{ \mu_f(M_{(k-1)^\theta}P_{\W_2, h_2})}\\
  &=\sum_{2|k} u\left(\frac{k-1}{K}\right) \frac{(2\pi^2)^2}{(k-1)^2} \sum_{0\leq i,j\leq N-1} \frac{h_1^{2i}h_2^{2j}}{i!j!}\left(-\frac{1}{8}(k-1)^{2\theta-1}
\right)^{i+j}\\
  & \quad \quad  \times \sum_{f\in H_k}\frac{A^{W_{1,i}}_f( (k-1)^{1-\theta},h_1)A^{W_{2,j}}_f((k-1)^{1-\theta},h_2)}{L(1, \sym^2 f)}\\
  & \quad \quad\quad+O_{V_1,V_2,h,N}( K^{1/2-3\theta/2+\eps}),
  \end{align}
with an allowed implied constant.  Now for each pair $i,j\in\{0,\ldots,N-1\}$, we apply Theorem \ref{QVthm} with smooth weights $W_{1,i}$ $W_{2,j}$ and weight function
  \begin{equation}u_{ij}(y)= u(y) y^{(2\theta-1)(i+j)-1}. \end{equation}
  This gives
  \begin{align}
  &\sum_{k\geq 1,2|k} u\left(\frac{k-1}{K}\right) \frac{(2\pi^2)^2}{(k-1)^2} \frac{h_1^{2i}h_2^{2j}}{i!j!}\left(-\frac{1}{8}(k-1)^{2\theta-1}\right)^{i+j}\\
  &\quad \quad\times \sum_{f\in H_k}\frac{A^{W_{1,i}}_f( (k-1)^{1-\theta},h_1)A^{W_{2,j}}_f((k-1)^{1-\theta},h_2)}{L(1, \sym^2 f)}\\
  &= 2\pi^2 h_1^{2i}h_2^{2j}\left(\frac{-1}{8}\right)^{i+j} K^{2(i+j)(\theta-1/2)-1}\sum_{k\geq 1,2|k} u_{ij}\left(\frac{k-1}{K}\right) \frac{2\pi^2}{(k-1)}\\
  &\quad \quad\times \sum_{f\in H_k}\frac{A^{W_{1,i}}_f( (k-1)^{1-\theta},h_1)A^{W_{2,j}}_f((k-1)^{1-\theta},h_2)}{L(1, \sym^2 f)}\\
  &= 2\pi^2 h_1^{2i}h_2^{2j}\left(\frac{-1}{8}\right)^{i+j}  \left(\int_0^\infty u_{ij}(y)y^{1-\theta}dy \right)\\
  & \quad\cdot B_{h_1,h_2}(W_{1,i}, W_{2,j})K^{1-\theta-(i+j)(1-2\theta)}
  +O_{h_i,V_i,\theta, N, \e}\left(K^{-(i+j)(1-2\theta)+\eps} \right),
  \end{align}
  with an implied constant of the desired form.
  For $(i,j)\neq (0,0)$ we see that the contribution is bounded by $O(K^{\theta})$, and for $i=j=0$, we get the wanted main term. So in this case we have an error of order $O(K^{\max{(\theta, 1/2-3\theta/2)}})$, which translates to the claimed $\delta_\theta$.

  Now assume that $\theta=1/2$, which implies that $\frac{k-1}{(n+h/2)^2}\asymp 1$  for non-zero terms in the sum \eqref{neverending}. Again by a Taylor expansion, we see that
  \begin{equation}   \left(\frac{\sqrt{n(n+h)}}{n+h/2} \right)^{k-1}= \exp\left(-\frac{h^2}{8}\frac{k-1}{(n+h/2)^2}\right)+O\left( \frac{h^4}{k}\right), \end{equation}
  which is the source of the different main term in this case. We proceed as above to write
  \begin{align*}
  &\sum_{k\geq 1,2|k} u\left(\frac{k-1}{K}\right)\sum_{f\in H_k}L(1,\sym^2 f) \mu_f(M_{(k-1)^\theta}P_{\W_1, h_1})\overline{ \mu_f(M_{(k-1)^\theta}P_{\W_2, h_2})}\\
  &=\sum_{k\geq 1,2|k} u\left(\frac{k-1}{K}\right) \frac{(2\pi^2)^2}{(k-1)^2} \sum_{f\in H_k}\frac{A^{W_1^*}_f((k-1)^{1-\theta},h_1)A^{W_2^*}_f((k-1)^{1-\theta},h_2)}{L(1, \sym^2 f)}\\
  &\quad\quad\quad+O_{V_i,h_i}( K^{-1/4+\eps}),
  \end{align*}
  where $W_i^*(y)=\W_i\left(\frac{1}{4\pi y }\right)\exp\left(-\frac{h_i^2}{8}/y^2\right)$ for $i=1,2$, and where the implied constant is of the desired form. Again by an application of Theorem \ref{QVthm}, we get the desired main term with error-term $O_{V_i,h_i}(K^\eps)$ and an implied constant of the desired form.

  Finally when $\theta>1/2$, we see that for non-zero terms in the sum \eqref{neverending}  we have $\frac{k-1}{(n+h/2)^2}\gg k^{2\theta-1}$, which implies
  \begin{equation}   \left(\frac{\sqrt{n(n+h)}}{n+h/2} \right)^{k-1}\ll  \exp\left(- c k^{2\theta-1}\right), \end{equation}
  for some $c>0$ depending only on $V$,  and hence we get exponential decay of the sum in \eqref{neverending}. Therefore we can even get the desired bound without any averaging. By summing up we arrive at the error-term $O(\norm{V_1}_\infty \norm{V_2}_\infty K^{-2\theta+\eps})$.
  \end{proof}

\begin{rem} The above theorem also holds, with the same proof, when we allow $\W_i$ to have support in $\R_+$, if we interpret $M_{(k-1)^\theta}P_{\W,m}$ as the Poincar\'e series $P_{\W_{k,\theta}}$ related to $\W_{k,\theta}=\W(y/(k-1)^\theta)$.
\end{rem}
\begin{rem}\label{remark:hneq0}
We now give a quick sketch of what happens in the case when  $h_1=0$ and  $\int_0^\infty \W_1(y) y^{-2}dy=0$ (i.e. in the case where $P_{\W_1,0}$ is an incomplete Eisenstein series orthogonal to 1) and $h_2\neq 0$. The translation to a shifted convolution sum as in \eqref{neverending} is still valid.

To analyze the resulting shifted convolution sum we imitate the proof of Theorem \ref{QVthm}. In this case we use the Hecke relations \eqref{hecke-relations} to write \begin{equation}A_f^{W_1}(X,0)=\sum_{d\in\N}\sum_{r\in \N}\lambda_f(r^2)W_1\left(\frac{dr}{X}\right).\end{equation} Here  $W_1(y)=\W_1(1/(4\pi y))$ and $X=(k-1)^{1-\theta}$.   We deal with the off-diagonal terms as above and the diagonal term from the Petersson formula becomes
\begin{equation} \sum_{\substack{d_1\in \N,d_2|h_2\\ r_1,r_2\in\N}}\delta_{r_1^2=r_2(r_2+d_2)}W_1\left(\frac{X}{d_1r_1} \right)W_2\left(\frac{X}{\frac{h_2}{d_2}(r_2+d_2/2)} \right). \end{equation}
 Now we observe that for fixed $d_2$ the equation  $r_1^2=r_2(r_2+d_2)$ has only finitely many solutions $(r_1,r_2)$ and for any such solution, we have by Poisson summation
\begin{align*} \sum_{d_1\in \N } W_1\left(\frac{X}{d_1r_1} \right)&=\int_0^\infty W_1\left(\frac{X}{r_1y} \right)dy+O_{h_2,A}(X^{-A})\\
&=\frac{X}{4\pi r_1}\int_0^\infty V_1\left(y \right)\frac{dy}{y^2}+O_{h_2,A}(X^{-A})=O_{h_2,A}(X^{-A}).
\end{align*}
Therefore the conclusion of  Theorem \ref{thm:calculating-variances} holds in this case with $B_\theta(P_{V_1,0},P_{V_2,h_2})=0$.\\

Now if $h_i=0$ and $\int_0^\infty V_i(y) y^{-2}dy=0$,  since the factor $\left(\frac{\sqrt{n(n+h_i)}}{n+h_i/2}\right)^{k-1}=1$, we do not have to distinguish between various regimes of $\theta$. Using a similar analysis we find that
\begin{align}
\sum_{2\mid k} &u\left(\frac{k-1}{K}\right)
  \sum_{f\in H_k}L(1,\sym^2 f)  \mu_f(M_{(k-1)^\theta}P_{\W_1,0}) \overline{\mu_f(M_{(k-1)^\theta}P_{\W_2,0})}\\
\label{coffee-coffee}      &=2\pi^2 \sum_{2\vert k}\frac{u\left(\frac{k-1}{K}\right)}{k-1}\sum_{r,d_1,d_2\in \N}V_1\left(\frac{(k-1)^{1-\theta}}{4\pi rd_1}\right)\overline{V_2\left(\frac{(k-1)^{1-\theta}}{4\pi rd_2}\right)}\\&\quad\quad +O(\max{(K^{1/2-3\theta/2+\e},1)}).
           \end{align}

Analogous to \cite[p. 781]{LuoSarnak:2004a}, by using successive Euler--Maclaurin summation on the $d_i$ sums, see \cite[Eq. (4.20)]{IwaniecKowalski:2004a}, followed by Poisson summation on the $r$ sum and on the $k$ sum we have that  in this case Theorem \ref{thm:calculating-variances} holds with
\begin{equation}B_\theta(P_{\W_1,0},P_{\W_2,0})=\frac{\pi}{4}\int_0^\infty \int_0^\infty \int_0^\infty b_2(y_1)b_2(y_2) \tilde \W_1\left(\frac{t}{y_1}\right)\tilde V_2\left(\frac{t}{y_2}\right)  \frac{dy_1}{y_1^2}\frac{dy_2}{y_2^2} \frac{dt}{t^2},\end{equation}
with error  term $O(\max{K^{1/2-3\theta/2+\e},1})$, and the same type of implied constant.
Here $\tilde \W_i(y)=(\W_i^\prime (y)y^2)^\prime$ and $2b_2(y)=B_2(y-\lfloor y \rfloor)$, where $B_2(y)=y^2-y+{1}/{6}$ is the second order Bernoulli polynomial. Note that the $y_i$ integrals vanishes for $t$ sufficiently small so the $t$-integral converges (although not absolutely).

\end{rem}

\section{Extension of \texorpdfstring{$B_\theta(\cdot{,}\cdot)$}{Btheta} and quantum variance for more general observables}
Let \begin{equation}C_{0}^\infty (M,B):=
\left\{ \psi:M\to \C \text{ smooth} \left\vert
\begin{array}{l}
\scriptstyle{
\supp \psi\subset B}\vspace{-5pt}\\
\scriptstyle{\psi \text{ decays rapidly at $\infty$}}\vspace{-5pt}\\
\scriptstyle{\int_0^1\psi(z)dx=0 \textrm{ for $y$ large enough}}
\end{array}
\right.\right\},\end{equation}
where $B=\{x+iy\in M \mid y> 1\}\subset X$ is the standard horocyclic region. In this section we will extend the above variance results to the space \begin{equation}
    C_{0,0}^\infty (M,B)=\{\psi\in C_{0}^\infty (M,B): \inprod{\psi}{1}=0\}.
\end{equation}

For $\psi\in C_{0,0}^\infty (M,B)$ we let $V^{\psi}_m$ be its $m$th Fourier coefficient.
Note that, since $\psi$ is supported in $B$,  the coefficient $V^{\psi}_m(y)$ is supported in $y>1$ and we have  \begin{equation}\label{psi-FE}
\psi(z)=\sum_{m\in\Z} V^{\psi}_m(y) e(mx) =\sum_{m\in \Z} P_{V^{\psi}_m,m}(z),
\end{equation}
where  $V^{\psi}_0$ has compact support, and satisfies $\int_0^\infty\W_0^{\psi}(y)y^{-2}\,{dy}=0$.
Inspired by Theorem \ref{thm:calculating-variances} and Remark \ref{remark:hneq0}, we define, for $\psi_1, \psi_2\in C_{0,0}^\infty (M,B)$,
the Hermitian form \begin{align}
B_\theta(\psi_1,\psi_2)=\frac{\pi}{4}&
\sum_{m, n\neq 0}\tau_1((\abs{m},\abs{n}))\int_0^\infty V^{\psi_1}_m\left(\frac{y}{\abs{m}}\right)\overline{V^{\psi_2}_n\left(\frac{y}{\abs{n}}\right)}f_{\theta,m,n}(y)\frac{dy}{y^2}\\
\label{Btheta}&+\frac{\pi}{4}\int_0^\infty \int_0^\infty \int_0^\infty b_2(y_1)b_2(y_2) \widetilde{V^{\psi_1}_0}\left(\frac{t}{y_1}\right)\overline{\widetilde{V^{\psi_2}_0}\left(\frac{t}{y_2}
\right)}
\frac{dy_1}{y_1^2}\frac{dy_2}{y_2^2} \frac{dt}{t^2}.
\end{align}

Note that if $\psi_1,\psi_2$ consist of a single Fourier coefficient, and, if this coefficient is not just of rapid decay but of compact support, then \eqref{Btheta} agrees with the result of Theorem \ref{thm:calculating-variances} and Remark \ref{remark:hneq0}.
To see that $B_\theta(\psi_1,\psi_2)$ is well defined we argue as follows. By smoothness and rapid decay of $\psi$ and  using integration by parts, we see that
\begin{equation}V^{\psi}_m(y)\ll_{A,B, \psi} y^{-A}m^{-B},\end{equation}
for any $A,B\geq 0$. It follows that
\begin{equation}\int_0^\infty V^{\psi_1}_m\left(\frac{y}{\abs{m}}\right)\overline{V^{\psi_2}_n\left(\frac{y}{\abs{n}}\right)}\frac{dy}{y^2} \ll_A (\abs{mn})^{-A},\end{equation}
and so  the first sum in \eqref{Btheta}  converges absolutely. The second term in \eqref{Btheta}  is well-defined by the discussion in Remark \ref{remark:hneq0}.

We observe that, when restricted to incomplete Eisenstein series, the form $B_\theta(\cdot{,}\cdot)$ is independent of $0<\theta<1$, while for cuspidal test functions $B_\theta(\cdot{,}\cdot)$ exhibits a phase transition at $\theta=1/2$  as claimed in Theorem \ref{intro:QuantumVariance} (ii).

We can now show that the variance result of Theorem \ref{thm:calculating-variances} can be extended to the space $C_{0,0}
^\infty(M,B)$.

\begin{thm} \label{qvargeneraltest}  Let $u:\R_+\to \R_{\geq 0}$ be a smooth compactly supported weight function, and let $\psi\in C_{0,0}^\infty (M,B)$ and $0<\theta<1$. Then we have
\begin{align}\label{qvarianceM,B}
\sum_{2|k}u\left(\frac{k-1}{K} \right)\sum_{f\in H_k}L(1,\sym^2 f) \abs{\mu_f(M_{(k-1)^\theta}\psi)-\nu{(M_{(k-1)^\theta}{\psi})}}^2 \\
=  B_\theta(\psi,\psi)\left(\int_0^\infty u(y)y^{-\theta} dy\right)K^{1-\theta}+O_{\psi,\theta}(K^{1-\theta-\delta_\theta}),
\end{align}
for $\delta_\theta>0$ as in Theorem \ref{thm:calculating-variances}.
\end{thm}
\begin{proof}
Consider a partition of unity
\begin{equation} \sum_{l\geq 0} u_l(y)=1_{\geq 1}(y)=\begin{cases}1,& y>1\\ 0, & y<1  \end{cases}, \end{equation}
where $u_l:\R_+\to [0,1]$ with $\supp u_l\subset (3^l,2\cdot 3^{l+1})$, $u_l$ smooth for $l>0$ and $u_0$ smooth on $(1,\infty)$ and $u_l^{(n)}(y)\ll_n y^{a_n}$ for some $a_n>0$ independently of $l$.
Multiplying this partition of unity on $\psi$ as in \eqref{psi-FE} we find
\begin{equation}\psi(z)=\W^\psi_0(y)+\sum_{l\geq 0,m\neq 0} \W^\psi_{l,m}(y)e(mx), \end{equation}
where $\W^\psi_{l,m}(y)=u_l(y)\W^\psi_m(y)$ and $\W_0^\psi(y)$ are smooth with compact support. We have
\begin{equation}\label{cool-bound} {\W^\psi_{l,m}}^{(n)}(y)\ll_{C} y^{-C}\abs{m}^{-C}, \end{equation}
for any $C>0$  and independent of $l$.  To see this  we note that by the definition and partial integration
\begin{equation}{\W^{\psi}_{l,m}}^{(n)}(y)=\frac{1}{(2\pi i m)^C}\sum_{j=0}^n \binom{n}{j} u_l^{(n-j)}(y)\int_0^1 \left(\frac{\partial^{C+j}}{\partial x^C \partial y^{j}}\psi\right)(z) e(-mx)dx.\end{equation}
Now by using the rapid decay of $\psi$ and the bound of the derivatives of $u_l$, we arrive at \eqref{cool-bound}. This implies, in particular, that for every $C\geq 0$ we have $\norm{ \W^\psi_{l,m}}_{C, \infty}\ll_{C, \psi}{3^{-Cl}\cdot\abs{m}^{-C}}$.

This implies, using Theorem \ref{thm:calculating-variances}, that for $m_1,m_2 $ and $l_1,l_2\geq 0$
\begin{align}
 \sum_{2|k}u\left(\frac{k-1}{K} \right)\sum_{f\in H_k} & L(1,\sym^2 f) \mu_f(M_{(k-1)^\theta}P_{\W_{l_1,m_1},m_1})\mu_f(M_{(k-1)^\theta}P_{\W_{l_2,m_2},m_2})\\
 &=B_\theta(P_{\W_{l_1,m_1},m_1},P_{\W_{l_2,m_2},m_2})\left(\int_0^\infty u(y)y^{-\theta} dy\right)K^{1-\theta}\\ &\quad\quad\quad\quad+O_{\psi,\theta}\left(\frac{K^{1-\theta-\delta_\theta}}{3^{l_1+l_2}((1+\abs{m_1})(1+\abs{m_2}))^{2}}\right).
\end{align}
Therefore by summing up all the contributions we get
\begin{align}\sum_{k,2|k}u&\left(\frac{k-1}{K} \right)\sum_{f\in H_k}L(1,\sym^2 f) |\mu    _f(M_{(k-1)^\theta}\psi)|^2 \\
&=  \left(\sum_{m_1,m_2,l_1,l_2}B_\theta(P_{h_{l_1,m_1},m_1},P_{h_{l_2,m_2},m_2})\right)\left(\int_0^\infty u(y)y^{-\theta} dy\right)K^{1-\theta}\\
&\quad\quad+O_{\psi,\theta}\left(K^{1-\theta-\delta_\theta}\left(\sum_{\substack{l_1,l_2>0\\m_1,m_2}}\frac{3^{-l_1-l_2}}{((1+\abs{m_1})(1+\abs{m_2}))^{2}}\right)\right)\\
&=B_\theta(\psi,\psi)\left(\int_0^\infty u(y)y^{-\theta} dy\right)K^{1-\theta}+O_{\psi,\theta}(K^{1-\theta-\delta_\theta}),\end{align}
which finishes the proof.
\end{proof}

\section{Small scale quantum ergodicity around infinity}

In this section we show that if we average  over $f\in H_k$ and over the weight $k$ quantum ergodicity holds for appropriately chosen sets shrinking towards the cusp all the way down to the Planck scale.

\begin{thm}\label{ergodicity}Let $0<\theta<1$ and $\psi\in C_{0}^\infty(M,B)$. Then
 \begin{align}
    \sum_{2\mid k} u\left(\frac{k-1}{K}\right) \sum_{f\in H_k}L(1,\sym^2 f) &\abs{\mu_f(M_{(k-1)^\theta}\psi)-\nu(M_{(k-1)^\theta}\psi)}^2\\&= O_{\psi, u} (K^{\max{(2-2\theta-1/5,1-\theta)}}).
\end{align}

  \end{thm}

\begin{proof}

Note that $\psi\in C_{0}^\infty(M,B)$ can be written as $\psi=\psi_0+\psi_1$, where $\psi_1\in C_{0,0}^\infty(M,B)$ and $\psi_0=P_{V,0}$ is an incomplete Eisenstein series with $V$ supported in $(1,\infty)$. Since trivially
\begin{equation}
    \abs{\mu_f(M_{(k-1)^\theta}\psi)-\nu(M_{(k-1)^\theta}\psi)}^2\leq 2\sum_{i=1,2}\abs{\mu_f(M_{(k-1)^\theta}\psi_i)-\nu(M_{(k-1)^\theta}\psi_i)}^2,
\end{equation}
we may use Theorem \ref{qvargeneraltest} to see that we only need to prove Theorem \ref{ergodicity} in the case where $\psi=P_{V,0}$, which we assume for the rest of the proof. In order to do so, we open up the square and compute asymptotics with error terms for each of the averages over each of the terms $\abs{\mu_f(M_{(k-1)^\theta}\psi)}^2$,  $\abs{\nu(M_{(k-1)^\theta}\psi)}^2$, ${\mu_f(M_{(k-1)^\theta}\psi)\overline{\nu(M_{(k-1)^\theta}\psi)}}$ and its conjugate. Since $\nu(M_{(k-1)^\theta}\psi)=(k-1)^{-\theta}\nu(\psi)$ we see that this essentially corresponds to computing the second, zero-th, and first moment of $\mu_f(M_{(k-1)^\theta}\psi)$.

  We start by showing that
\begin{align}\label{paper}
    \sum_{2\mid k} u&\left(\frac{k-1}{K}\right) \sum_{f\in H_k}L(1,\sym^2 f) \abs{\mu_f(M_{(k-1)^\theta}\psi)}^2\\
    &=\abs{\nu(\psi)}^2\frac{\zeta(2)^2}{12}\int_0^\infty u(y)y^{1-2\theta}dy\frac{K^{2-2\theta}}{2}
\label{expansion-1}    +O_{\psi,u} (K^{1-\theta}).
\end{align}
To prove this we start as in Remark \ref{remark:hneq0} and arrive at \eqref{coffee-coffee}. We then evaluate the sum over  $d=d_i$ using  the second order Euler--Maclaurin formula and find that
  we have for any $X>0$
\begin{align}\label{EulerMac}\sum_{d} V\left(\frac{X}{ rd }\right)&= \int_0^\infty V\left(\frac{X}{ ry }\right)dy -\int_0^\infty b_2(y)\tilde V\left(\frac{X}{ r y }\right)\frac{dy}{y^2}\\
&= \frac{X}{r}\int_0^\infty V(y)\frac{dy}{y^2} -\int_0^\infty b_2(y)\tilde V\left(\frac{X}{ r y }\right)\frac{dy}{y^2},\end{align}
  where $2b_2(y)=B_2(y-\lfloor y \rfloor)$ and $B_2(y)=y^2-y+{1}/{6}$ is the second Bernoulli polynomial and $\tilde V(y)=(V^\prime (y)y^2)^\prime $. Here we have used that  $\frac{\partial^2}{\partial y^2} V\left(\frac{X}{ ry }\right)= \tilde V\left(\frac{X}{ ry }\right) y^{-2}$.
  We know by the assumptions on $V$ that the above defines a smooth function in $r$ and that $\sum_{d} V\left(\frac{X}{ dr }\right)$ vanishes for $r>AX$.

  We can now evaluate \begin{equation}
      \sum_{r,d_1,d_2\in\N}V\left(\frac{X}{ rd_1 }\right)\overline{V\left(\frac{X}{ rd_2 }\right)}
  \end{equation}
  by inserting \eqref{EulerMac} and evaluating the four terms coming from opening the square.  The contribution coming from the absolute square of the first term on the right of \eqref{EulerMac} equals
  \begin{align*}
 X^2\abs{\int_0^\infty V(y)\frac{dy}{y^2}}^2\sum_{1\leq r \leq AX} \frac{1}{r^2}
 =X^2\abs{\int_0^\infty V(y)\frac{dy}{y^2}}^2 \zeta(2)+O(X).
  \end{align*}
A change of variables combined with the fact that $b_2(v)$ is uniformly bounded shows that $\int_0^\infty b_2(y)\tilde V\left(\frac{X}{ r y }\right)\frac{dy}{y^2}\ll_V r/X $. This implies that the remaining contributions are $O(X)$. Plugging these estimates back in \eqref{coffee-coffee} with $X=(k-1)^{1-\theta}/4\pi$ and using Poisson summation in the $k$ variable. we complete the proof of \eqref{expansion-1}.

 We next show that
  \begin{align}\label{rock}
  \sum_{2\vert k} u\left(\frac{k-1}{K}\right)\sum_{f\in H_k}L(1, \sym^2 f)=\frac{\zeta(2)^2
  }{12}\frac{K^2}{2}\int_0^\infty u(y)ydy+O(K^{2-\frac{1}{5}+\e}).
  \end{align}
 To approximate $L(1, \sym^2 f)$ we use $e^{-x}=\frac{1}{2\pi i}\int_{(\sigma)}\Gamma(s)x^{-s}ds$ to see that
  \begin{equation}
      \sum_{n=1}^\infty\frac{\lambda_f(n^2)}{n}e^{-n/T}=\frac{1}{2\pi i }\int_{(2)}\Gamma(s)\frac{L(s+1, \sym^2 f)}{\zeta(2(s+1))}T^sds.
  \end{equation}
 Here $T\geq 1$ is a parameter which will be chosen later. For now we assume that $T=K^a$ with $1<a<2$. Moving the line of integration to $\sigma=-1/2$ we pick up a pole of the Gamma function at $s=0$ and we find that
   \begin{equation}\label{approxim-symm2-at-1}
      \sum_{n=1}^\infty\frac{\lambda_f(n^2)}{n}e^{-n/T}=\frac{L(1, \sym^2 f)}{\zeta(2)}+I_f(T),
  \end{equation}
  where $I_f(T)=\frac{1}{2\pi i }\int_{(-1/2)}\Gamma(s)\frac{L(s+1, \sym^2 f)}{\zeta(2(s+1))}T^sds$. Using any bound of the form
\begin{equation}L(s,\sym^2 f)\ll_A (1+\abs{s})^A(k^{2})^{1/4-\rho}\end{equation}
for $\Re(s)=1/2$ we see, that $I_f(T)\ll_A T^{-1/2}k^{1/2-2\rho+\e}$. In fact the convexity estimate $\rho=0$ will suffice for what we need. We have
\begin{equation}
    \sum_{n=1}^\infty\frac{\lambda_f(n^2)}{n}e^{-n/T}=\sum_{n\leq T^{1+\epsilon}}\frac{\lambda_f(n^2)}{n}e^{-n/T}+O_A(K^{-A}),
\end{equation} for any $A>0$. We observe also that since $\lambda_f(n^2)\ll n^\e$ we have $\sum_{n\in \N}\frac{\lambda_f(n^2)}{n}e^{-n/T}\ll T^\e.$
Using these observations we see that
\begin{align}
    \sum_{2\vert k} & u\left(\frac{k-1}{K}\right)\sum_{f\in H_k}L(1, \sym^2f)= \\
\label{freedom}&=\zeta(2)^2\sum_{2\vert k} u\left(\frac{k-1}{K}\right)\sum_{n_1, n_2\leq T^{1+\e}}\frac{e^{-(n_1+n_2)/T}}{n_1n_2}\sum_{f\in H_k}\frac{\lambda_f(n_1^2)\lambda_f(n_2^2)}{L(1, \sym^2 f)}\\
&\quad\quad + O\left(\sum_{k}u\left(\frac{k-1}{K}\right)\sum_{f\in H_k}K^\epsilon (\abs{I_f(T)}+\abs{I_f(T)}^2) +K^{-A}\right).
\end{align}
Using the convexity bound ($\rho=0$) for  $I_f(T)$ we see that the error is $O\left(K^{2+\e} \left(\left(\frac{K}{T}\right)^{1/2}+ \frac{K}{T}\right)\right)$.
Up to this error term the sum we want to estimate therefore equals
\begin{equation}
\frac{\zeta(2)^2K}{2\pi^2}\sum_{2\vert k} \tilde u\left(\frac{k-1}{K}\right)\sum_{n_1, n_2\leq T^{1+\e}}\frac{e^{-(n_1+n_2)/T}}{n_1n_2}\frac{2\pi^2}{(k-1)}\sum_{f\in H_k}\frac{\lambda_f(n_1^2)\lambda_f(n_2^2)}{L(1, \sym^2 f)},
\end{equation}
where  $\tilde u(y)=u(y)y$. We now use the Petersson  formula \eqref{Petersson} on the last sum. The diagonal term gives the claimed main term
\begin{align}
    \frac{\zeta(2)^2K}{2\pi^2}\sum_{2\vert k} \tilde u\left(\frac{k-1}{K}\right)\sum_{n_1 \leq T^{1+\e}}\frac{e^{-2n_1/T}}{n_1^2}=\frac{\zeta(2)^3}{2\pi^2}\frac{K^2}{2}\int_0^\infty u(y)ydy +O({K^2}/{T}).
\end{align}
We also need to bound the non-diagonal contribution which is done as in the proof of Theorem \ref{QVthm}. This consists of a $k$ sum with $k$ supported around $K$, sums over $n_1, n_2\leq T^{1+\e}$, and a $c$-sum. The $c$-sum can be truncated at $c\leq M$ at the expense of an error which is big $O$ of
\begin{align}
    K\sum_{k\asymp K}\sum_{n_1,n_2\leq T^{1+\e}}\frac{1}{n_1n_2}\sum_{c>M}\left(\frac{e\Delta}{2kc}\right)^{k-1}\ll K\sum_{k\asymp K}T^{\epsilon}\left(\frac{e4\pi T^{2+2\e} }{2kM}\right)^{k-1} \frac{M}{K},
\end{align}
where $\Delta=4\pi n_1n_2\leq 4\pi T^{2+2\e}$ and we have used \eqref{good-bound-Bessel} on the Bessel function. If we choose $M=C T^{2+2\e}K^{-1+\epsilon}$ for a suitably big constant $C$ the parenthesis is $\ll K^{-\e(k-1)}$, which decays exponentially so this contribution is $O_A(K^{-A})$  for every positive $A$.

By using \eqref{first}, as in the proof of Theorem \ref{QVthm}, we see that it suffices to bound
\begin{equation}
    K\sum_{n_1,n_2\leq T^{1+\e}}\frac{1}{n_1n_2}\sum_{c\leq M}\abs{\int_{-\infty}^\infty \hat{g}(t) \sin \left( \frac{\Delta}{c}  \cos(2\pi t)\right)dt}
\end{equation}
with $g(y)=\tilde u(y/K)$. Here it is clear that $g$ is supported in $y\asymp K$ and $g^{(m)}(y)\ll K^{-m}$ and we conclude as in \eqref{ref:zombie} that
\begin{equation}
\int_{-\infty}^\infty\abs{\hat g(t)t^m}dt\ll K^{-m}.
\end{equation}
We use \eqref{taylorsin} with $N=1$ and we estimate the contribution from the error terms by
\begin{equation}
    K\sum_{n_1,n_2\leq T^{1+\e}}\frac{1}{n_1n_2}\sum_{c\leq M}\int_{-\infty}^\infty \abs{\hat{g}(t)} \left(\frac{\Delta}{c}\right)^\alpha \abs{t}^\beta dt\ll K^{1-\beta} \left(\sum_{n_1\leq T^{1+\e}}n^{\alpha-1}\right)^2 \sum_{c\leq M}c^{-\alpha}.
\end{equation}
For the four contributions $(\alpha, \beta)=(2,8)
, (4, 16), (0,2), (0,4)$ this gives an error term of $(K((T/K^2)^4+(T/K^2)^8)+T^2/K^2+T^2/K^{4})K^\e$, which are all less that $(K+T^2/K^2)K^\e$.
To bound the contribution involving $e(\frac{\Delta}{c},t)$ (recall definition \eqref{ref:star}) we see as in \eqref{ref:starstar} that
\begin{equation}\label{uno}
\frac{d^n}{dy^n}\left(g^{(m)}\left(\sqrt{\frac{2\Delta}{c}y}\right) y^{-1/2}\right)\ll_{u ,n,m} {K^{-m} y^{-1/2-n}},
\end{equation}
so again we find
\begin{equation}\label{duo}
     \int_{-\infty}^\infty\hat g(t)e(\frac{\Delta}{c},t)dt \ll_u
     \frac{K^{-3}\Delta^{1/2}}{c^{1/2}}+\frac{K^{1-2n}\Delta^{n-1/2}}{c^{n-1/2}}.
 \end{equation}
 It turns out to be convenient to interpolate the estimates for  $n=1$ and $n=2$ (using that $\min(a,b)\leq a^\lambda b^{1-\lambda}$ for $a,b\geq 0$, $0\leq \lambda\leq 1$) and use $n=3/2$ such that the last contribution is
\begin{align}
    K\sum_{n_1,n_2\leq T^{1+\e}}\frac{1}{n_1n_2}\sum_{c\leq M}\int_{-\infty}^\infty\hat g(t)e(\frac{\Delta}{c},t)dt\ll T^{2}K^{-5/2+\e}+K^{-1+\e}T^2.
\end{align}
The total error therefore become  $\ll K^{2+\e} \left(\frac{K}{T}\right)^{1/2} + K^{1+\e}  + T^{2}K^{-1+\e}$, as all other contributions are smaller. Choosing $T=K^{7/5}$ we complete the proof of  \eqref{rock}.

Lastly we use a similar strategy to prove that
  \begin{align}\label{anotherone}
  \sum_{2\vert k} &u(\frac{k-1}{K})\sum_{f\in H_k}L(1, \sym^2 f)\mu_f(M_{(k-1)^\theta}P_{V,0})\\&=\nu(P_{V,0})\int_0^\infty u(y)y^{1-\theta}dy\frac{\zeta(2)^2
  }{12}\frac{K^{2-\theta}}{2}+O(K^{2-\theta -(1/4+3\theta/8)+\e}+K^{1+\e}).
  \end{align}
   We use \eqref{approxim-symm2-at-1} to approximate $L(1, \sym^2 f)$ by  $\zeta(2)\sum_{n\leq T^{1+\e}}\frac{\lambda_f(n^2)}{n}e^{-n/T}$  at the cost of an error satisfying  $\ll K^{2-\theta+\e}\frac{K^{1/2-2\rho}}{T^{1/2}}$. We then use \eqref{neverending} and  the Hecke relations \eqref{hecke-relations} to arrive at
  \begin{equation}\zeta(2)\sum_{2\vert k}\!\!u\left(\frac{k-1}{K}\right)\!\!\!\!\sum_{\substack{n_1\leq T^{1+\e}\\n_2\in \N\\d\vert n_2}}\!\!\!\!\frac{e^{-n_1/T}}{n_1}V\left(\frac{(k-1)^{1-\theta}}{4\pi n_2}\right)\frac{2\pi^2}{(k-1)}\sum_{f\in H_k}\frac{\lambda_f(n_1^2)\lambda_f(d^2)}{L(1, \sym^2 f)},
 \end{equation}
  at the expense of an additional error which is $\ll K^{1-\theta+\e}.$ We then use the Petersson formula \eqref{Petersson}. The diagonal gives
  \begin{equation}
  \zeta(2)\sum_{2\vert k}u\left(\frac{k-1}{K}\right)\sum_{n_1\leq T^{1+\e}}\frac{e^{-n_1/T}}{n_1}\sum_{r=1}^\infty V\left(\frac{(k-1)^{1-\theta}}{4\pi rn_1}\right),
  \end{equation}
  which, after using Poisson summation in the $r$ variable, a change of variables, and then Poisson summation again in the $k$ variable,  gives the claimed main term up to an error which is $\ll K^2/T+K^{1+\e}$.

  The off-diagonal is handled as before: We truncate the $c$-sum at $M=CK^{-\theta+\e} T^{1+\e}$ with $C$ sufficiently large at the expense  $\ll K^{-A}$. We then use \eqref{first} with \begin{equation}g(y)=u(y/K)v\left(\frac{y^{1-\theta}}{4\pi n_2}\right)\end{equation} and find that in the support of the sums $g^{(m)}(y)\ll K^{-m}$, which allows us to bound the error coming from the approximation of $\sin(\frac{\Delta}{c}\cos(2\pi t))$ with $N=1$, $\Delta=4\pi n_1n_2$ as  \begin{equation}\ll
      T^{2}K^{(1-\theta)3-8+\e}+
      T^{4}K^{(1-\theta)5-16+\e}+
      TK^{-1-2\theta+\e}+
      TK^{-3-2\theta+\e}.
  \end{equation}
  We also find, using Fa\`a di Bruno's formula as before, that analogues of \eqref{uno} and \eqref{duo} hold for this $g$. Using \eqref{duo} with $n=2$ we get the final error contribution to be bounded by $TK^{-3/2-2\theta+\e}+ T^{3/2}K^{-1/2-5\theta/2}$. Balancing $T^{3/2}K^{-1/2-5\theta/2}=K^{2-\theta}\frac{K^{1/2}}{T^{1/2}}$ gives $T=K^{3/2-3\theta/4}$. This proves \eqref{anotherone} as with this choice of $T$ all error contributions are less than the claimed one.

 We can now finish the proof: We open up the square of the expression on the right-hand side of the theorem and use the expressions in \eqref{paper}, \eqref{rock}, and \eqref{anotherone}. The main terms cancel and we are left with the claimed error term.

 \end{proof}

\begin{rem}
It is obvious from the above proof that a subconvexity result in the $k$-aspect for $L(s, \sym^2 f)$ would give an improvement of the exponent. In fact a non-trivial bound on the second moment of $L(s, \sym^2 f)$ in the weight aspect would suffice. For $s=1/2$ such estimates has been proved in \cite{Khan:2010a}.
\end{rem}
Theorem \ref{ergodicity} shows that if $0<\theta<1$ then mostly (i.e. in a full-density set of $f\in H_k$ ) we have
\begin{equation}
    \mu_f(M_{(k-1)^\theta}\psi)=\nu(M_{(k-1)^\theta}\psi)+o(k^{-\theta}).
\end{equation}
If we go below the Planck scale, i.e. if we let $\theta\geq 1$,  then this is \emph{not} the case i.e. mass equidistribution fails.
  \begin{prop}
  Let $\theta \geq 1$ and let $V:\R_+\to \R$ be a smooth function with compact support in $(1, \infty)$, which satisfies $\int_0^\infty V(y)dy/y^2\neq 0$ and let $\psi_V$ be the associated incomplete Eisenstein series. Then
  \begin{equation} \mu_f(M_{(k-1)^\theta}\psi_V)=o(\nu(M_{(k-1)^\theta}\psi_V )),  \end{equation}
  as $k\to \infty$. This means in particular that mass equidistribution fails for shrinking annuli around infinity below the Planck scale i.e. when $\theta\geq 1$.
  \end{prop}
  \begin{proof}
 We use \eqref{neverending} and observe that the sum is identically zero,  since $(k-1)^{1-\theta}/(4\pi n)$ is less than one which is outside the support of  $V$. Therefore
  \begin{equation} \mu_f(M_{(k-1)^\theta}\psi_V)= O_\eps(k^{-1-\theta+\eps}), \end{equation}
  and, since $\nu(M_{k^\theta}\psi_V )\asymp  k^{-\theta}$, the proposition follows. Alternatively, one can directly estimate the Fourier expansion of $f$.
 \end{proof}

\section{Further extensions of  \texorpdfstring{$B_\theta(\cdot{,}\cdot)$}{Btheta} and computations at truncated eigenfunctions.}
Before we extend $B_\theta(\cdot{,}\cdot)$ we notice that on the set $C_{0,0}^\infty(M,B)$, $B_\theta(\cdot{,}\cdot)$ is symmetric with respect to the Laplacian.
\begin{lem}\label{first-symmetric}
The map $B_\infty: C_{0,0}^\infty(M,B) \times C_{0,0}^\infty(M,B) \to \C$ satisfies
$B_\infty(\Delta \psi,\varphi)=B_\infty( \psi,\Delta \varphi).$
\end{lem}

\begin{proof}
Writing $\psi$ as in \eqref{psi-FE} we note that $\Delta P_{V_m^{\psi},m}(z)=P_{{L_mV_m^\psi},m}(z)$, where $L_m= y^2\frac{d^2}{dy^2}-4\pi^2m^2 y^2$ and that the support of $L_m V_m^\psi$ is contained in $(1,\infty]$ if this is the case for $V_m^\psi$.
 The argument is now a straightforward modification of \cite[p.~782]{LuoSarnak:2004a}.\end{proof}

We now extend $B_\theta(\psi_1,\psi_2)$ defined in \eqref{Btheta} on $C_{0,0}^\infty(M,B)$ to the slightly larger space $1_B C_{0,0}^\infty(M)$. This space includes for instance $1_B\cdot\phi$ where $\phi$ is a Hecke--Maass form, which together with the incomplete Eisenstein series of mean 0 actually span this entire space. We may define $B_\theta(\psi_1,\psi_2)$ on this slightly larger space by the same formula \eqref{Btheta}. The same arguments as after \eqref{Btheta} shows that the infinite sum converges.

Unfortunately we do not know how to extend Lemma \ref{first-symmetric} to this larger space. When trying to do the obvious generalization we are faced with certain boundary terms that we cannot dismiss. This means also that, contrary to the situation when $\theta=0$ studied by Luo and Sarnak \cite{LuoSarnak:2004a}, we do not know if truncated  Hecke--Maass forms diagonalize $B_\theta(\cdot{,}\cdot)$ for $\theta>0$.

Consider the  subspace $C_{\rm cusp}^\infty(M,B)\subset C_{0,0}^\infty(M,B)$ consisting of functions where the zero-th Fourier coefficient vanishes completely. Note also that functions in $1_BC_{\rm cusp}^\infty(M,B)\subset 1_BC_{0,0}^\infty(M)$ also have zero-th Fourier coefficients vanishing completely.  We make the following analysis. It is straightforward to check that the Sobolev norm on $1_BC_{\rm cusp}^\infty(M)$ defined by
\begin{equation} \label{sobolev-again}
    \norm{1_B\psi}^2_{2,N}=\sum_{j\leq N} \norm{1_B\frac{d^j\psi}{dx^j}}^2_{L^2(M)}
\end{equation} is indeed a norm. Note that for $1_B\psi\in C_{\rm cusp}^\infty(M)$ we may write
\begin{equation}\psi=\sum_{n\neq 0} V_n^{(\psi)}(y)e(nx),\end{equation} and we have
\begin{align}\label{sobolev-calc}\norm{1_B\frac{d^j\psi}{dx^j}}^2_{L^2(M)}=\sum_{n\neq 0}\abs{2\pi n}^{2j}\int_{0}^\infty\abs{1_{y\geq 1} V_m^{(\psi)}}^2\frac{dy}{y^2}.
\end{align}
\begin{prop}\label{continuity}\,
\begin{enumerate}
    \item For each $N=0,1,\ldots$ the set $C_{\rm cusp}^\infty(M,B)$ is dense in $1_BC_{\rm cusp}^\infty(M)$ with respect to $\norm{\cdot }_{2,N}$.
    \item There exist an absolute constant $c>0$ such that for $1_B\psi_i\in 1_BC_{\rm cusp}^\infty(M)$
\begin{equation}\abs{B_\theta(1_B\psi_1,1_B\psi_2)}\leq c \norm{1_B\psi_1 }_{2,1}\norm{1_B\psi_2 }_{2,1}.
\end{equation}
\item The form $B_\theta(\cdot{,}\cdot)$ is continuous on $1_BC_{\rm cusp}^\infty(M)\times 1_BC_{\rm cusp}^\infty(M)$ with respect to $\norm{\cdot }_{2,1}$.
\end{enumerate}
\end{prop}

\begin{proof}
To see that $C_{\rm cusp}^\infty(M,B)$ is dense in $1_BC_{\rm cusp}^\infty(M)$ we approximate $1_B$ by a smooth cut-off as follows: Fix $w:\R\to \R_{\geq 0}$ smooth and supported in $[1/2,1]$ with $\int_0^1 w(t)dt=1$. For $0<\delta<1/2$ we define
$w_{\delta}(t):=\delta^{-1}w( t/\delta)$.
This is supported in $[\delta/2,\delta]$ and satisfies $\int_{0}^1 w_\delta(t)dt=1$. We then define the function $1_B^{\delta}:\R_+ \to \R$ as the convolution of $1_{y>1}$ and $w_\delta$ i.e.
\begin{equation} 1_B^{\delta}(y):=\int_0^\infty 1_{t>1}(t) w_\delta(y-t)dt. \end{equation}
We observe that $1_B^{\delta}$ is smooth and supported in $[1+\delta/2,\infty]$. It satisfies $0\leq 1_B^{\delta}(y)\leq 1$ and $1_B^{\delta}(y)=1$ for $y\geq 1+\delta$.

Let now $1_B\psi\in 1_BC_{\rm cusp}^\infty(M)$, and observe that $1_B^{\delta}\psi\in C_{\rm cusp}^\infty(M,B)$, where we use the same notation for $y\mapsto 1_B^{\delta}(y)$ and $x+iy\mapsto 1_B^{\delta}(y)$.
Furthermore
\begin{align}
    \norm{1_B\psi -1_B^\delta\psi}_{2,N}&= \norm{1_B(\psi -1_B^\delta\psi)}_{2,N}\\
    \label{sobolev-thingy}& \leq \max_{\substack{1\leq \Im(z)\leq 2\\ j\leq N} }\abs{\frac{d^j\psi}{dx^j}(z)}\sqrt{N+1}\int_{1}^2\abs{1-1_B^\delta(y)}\frac{dy}{y^2},
\end{align} which tends to zero as $\delta\to 0$. Since $1_B^\delta\psi\in C_{\rm cusp}^\infty(M,B)$ this proves that $C_{\rm cusp}^\infty(M,B)$ is dense in $1_BC_{\rm cusp}^\infty(M)$ with respect to $\norm{\cdot }_{2,N}$.

To prove the inequality for $B_\theta(\cdot{,}\cdot)$ we see from \eqref{Btheta},  the bound $\tau_1((\abs{m}, \abs{n}))\ll_\e \abs{mn}^{1+\epsilon}$, and Cauchy--Schwarz on the involved integral that for $1_B\psi_i\in 1_BC^\infty_{\rm cusp}(M)$ we have that $\abs{B_\theta(1_B\psi_1,1_B\psi_2)}$ is bounded by a constant times
\begin{equation}
    \sum_{m, n\neq 0}\abs{mn}^{1+\e}\left(\int_0^\infty \abs{1_{y/\abs{m}\geq 1}V^{(\psi_1)}_m\left(\frac{y}{\abs{m}}\right)}^2\frac{dy}{y^2}\int_0^\infty \abs{1_{y/\abs{n}\geq 1}V^{(\psi_2)}_n\left(\frac{y}{\abs{n}}\right)}^2\frac{dy}{y^2}\right)^{1/2}.
\end{equation}
Doing a change of variables this splits as a product of
\begin{equation}
\sum_{m\neq 0}\abs{m}^\e\left(\int_0^\infty \abs{1_{y\geq 1}V^{(\psi_1)}_m\left(y\right)}^2\frac{dy}{y^2}\right)^{1/2}
\end{equation}
times the same expression for $\psi_2$. Dividing  and multiplying the terms by $\abs{m}^{1/2+\e}$ we can use the Cauchy--Schwarz inequality to see that this is bounded by
\begin{equation}\left(\sum_{m\neq 0}\frac{1}{\abs{m}^{1+2\e}}\right)^{1/2}\left(
\sum_{m\neq 0}\abs{m}^{1+4\e}\int_0^\infty \abs{1_{y\geq 1}V^{(\psi_1)}_m\left(y\right)}^2\frac{dy}{y^2}\right)^{1/2}.
\end{equation}
Comparing with \eqref{sobolev-thingy} and \eqref{sobolev-again} we see that this is bounded by a constant times $\norm{1_B\phi_1}_{2,1}$, which proves the inequality for $B_\theta(\cdot{,}\cdot)$.

To see that $B_\theta(\cdot ,\cdot )$ is continuous on $1_BC_{\rm cusp }^\infty(M)$  we observe that, if we consider the sequence $(1_B\psi_{1,n},1_B\psi_{2,n})\to (1_B\psi_1,1_B\psi_2)$ with respect to $\norm{\cdot }_{2,1}$ as $n\to\infty$, then we can use that
\begin{align}
B_\theta(1_B\psi_{1,n},1_B\psi_{2,n})&-B_\theta(1_B\psi_1,1_B\psi_2)\\
&=B_\theta(1_B\psi_{1,n}-1_B\psi_1,1_B\psi_{2,n})+B_\theta(1_B\psi_1,1_B\psi_{2,n}-1_B\psi_2),
\end{align}
and the claim now follows easily from the inequality satisfied by $B_\theta(\cdot{,}\cdot)$.

\end{proof}

If $\phi$ is a Hecke--Maass form then $1_B\phi\in 1_BC_{\rm cusp}^\infty(M)$ and we consider   the expansion \begin{equation}
1_B\phi(z)=\sum_{m\neq 0} P_{1_{y>1}a_m^{(\phi)},m}(z).
\end{equation}
with $a_m^{(\phi)}(y)=\epsilon_{\phi,m}2\lambda_{\phi} (\abs{m})y^{1/2} K_{s-1/2}(2\pi \abs{m} y)$ and
where $\epsilon_{\phi,m}=1$, if $\phi$ is an even and $\epsilon_{\phi,m}=\sgn{(m)}$, if $\phi$ is odd.
It follows from this and \eqref{Btheta} that  $B_\theta(1_B\phi_1,1_B\phi_2)=0$, if either $\phi_1$ or $\phi_2$ is odd. This is also the case when $\theta=0$ as proved in \cite{LuoSarnak:2004a} as follows from  $L(\phi,1/2)=0$ for $\phi$ odd. If $\phi_1,\phi_2$ are both even Hecke--Maass forms with Laplace eigenvalues $s_1(1-s_1)$ and $s_2(1-s_2)$, respectively, we see that $ B_\theta(1_B\phi_1,1_B\phi_2)$ equals
\begin{equation}4\pi \sum_{m,n\geq 1}  \frac{\tau_1((m,n))\lambda_{\phi_1}(m)\lambda_{\phi_2}(n)}{(mn)^{1/2}} \int_{\max(m,n)}^\infty K_{s_1-1/2}(2\pi y) \overline{K_{s_2-1/2}(2\pi y)}\frac{dy}{y}, \end{equation}
for $0<\theta<1/2$ and for $\theta=1/2$, the number $B_{1/2}(1_B\phi_1,1_B\phi_2)$ equals
\begin{equation} 4\pi \sum_{m,n\geq 1}  \frac{\tau_1((m,n))\lambda_{\phi_1}(m)\lambda_{\phi_2}(n)}{(mn)^{1/2}} \int_{\max(m,n)}^\infty \!\!\!\!\!\!\!\!\!\!\!\!\!\!K_{s_1-1/2}(2\pi y)\overline{K_{s_2-1/2}(2\pi y)} e^{-2\pi^2 y^2(m^2+n^2)}\frac{dy}{y}, \end{equation}
as claimed in Theorem \ref{intro:QuantumVariance} (iv).

It is a deep number-theoretic fact that the central value of $L(\phi_j,s)$ is non-negative. Luo and Sarnak  \cite{LuoSarnak:2004a}  observed that this follows from noticing  that these numbers are essentially the eigenvalues of the non-negative Hermitian form $B_0$. We are now ready to draw a similar conclusion for $B_\theta(1_B\phi,1_B\phi)$ as computed above from the fact that $B_\theta(\cdot{,}\cdot)$ is non-negative on   $C_{0,0}^\infty(M,B)$ for any $0\leq \theta <1$ . Since we only know beforehand that $B_\theta(\cdot{,}\cdot)$ is non-negative on the smaller space  $C_{0,0}^\infty(M,B)$, we use the continuity properties of $B_\theta(\cdot, \cdot)$.

\begin{proof}[Proof of Corollary \ref{geq0}] We have seen above that the expression on the right of Corollary \ref{geq0} equals, up to a positive constant, the value $B_{\theta}(1_B\phi, 1_B\phi)$. It follows from Proposition \ref{continuity} there exist $\{\psi_n\}\subset C_{\rm cusp}(M,B)$ such that $\psi_n\to 1_B\phi$ with respect to $\norm{\cdot}_{2,1}$. By Theorem \ref{qvargeneraltest} we may conclude, since the left-hand side of \eqref{qvarianceM,B} is non-negative, that $B_\theta(\psi_n, \psi_n)\geq 0$. By the continuity properties of $B_\theta(\cdot{,}\cdot)$ in Proposition \ref{continuity} we conclude that $B_{\theta}(1_B\phi, 1_B\phi)\geq 0$, which proves the result.\end{proof}
Of course one may make a conclusion  analogous to that of Corollary \ref{geq0} for the case $\theta=1/2$, where the integrand gets multiplied by $e^{-2\pi^2y^2(m^2+n^2)}$.

\bibliographystyle{amsplain}
\providecommand{\bysame}{\leavevmode\hbox to3em{\hrulefill}\thinspace}
\providecommand{\MR}{\relax\ifhmode\unskip\space\fi MR }
\providecommand{\MRhref}[2]{%
  \href{http://www.ams.org/mathscinet-getitem?mr=#1}{#2}
}
\providecommand{\href}[2]{#2}

\end{document}